\documentclass{amsart}
\usepackage{graphicx}
\usepackage{amssymb}
\usepackage{enumerate}
\usepackage{mathrsfs}
\usepackage{url}

\numberwithin{equation}{section}
\numberwithin{figure}{section}
\theoremstyle{plain}
\newtheorem{thm}{Theorem}[section]

\newtheorem*{thm*}{Theorem}

\theoremstyle{definition}
\newtheorem{definition}[thm]{Definition}

\newtheorem*{claim}{Claim}
\theoremstyle{plain}
\newtheorem{proposition}[thm]{Proposition}
\theoremstyle{remark}
\newtheorem{remark}[thm]{Remark}
\theoremstyle{plain}
\newtheorem{corollary}[thm]{Corollary}
\theoremstyle{plain}
\newtheorem{lemma}[thm]{Lemma}
\theoremstyle{remark}
\newtheorem*{rem*}{Remark}

\newcommand{\abs}[1]{\left\vert#1\right\vert}
\newcommand{\set}[1]{\left\{#1\right\}}
\newcommand{\Real}{\mathbb{R}}

\newcommand{\Cpx}{\mathbb{C}}

\newcommand{\eps}{\varepsilon}
\newcommand{\half}{\frac{1}{2}}
\newcommand{\sph}{\mathbb{S}}
\newcommand{\To}{\rightarrow}
\newcommand{\pd}{\partial}

\newcommand{\hF}[1]{\mathsf{F}\left(#1\right)}

\newcommand{\Jp}{\mathsf{P}}

\newcommand{\acloc}{\mathfrak{AC}_{\text{loc}}}

\newcommand{\calL}{\mathcal{L}}

\title[Martin compactification of a negatively curved surface]{Martin compactification of a complete surface with negative curvature}
\author{Huai-Dong Cao}
\address{
Huai-Dong Cao \\ Department of Mathematics, University of Macau, Macau, China \&
Department of Mathematics, Lehigh University \\
Bethlehem, PA, 18015, USA}
\email{huc2@lehigh.edu}

\author{Chenxu He}
\address{Chenxu He \\
Department of Mathematics, University of Oklahoma \\
Norman, OK 73072}
\email{che@math.ou.edu}

\subjclass[2000]{53C21, 58G20, 31C35, 33C05, 34B25\\
The research of the first author was partially supported by NSF Grant DMS-0909581.}

\begin{document}
\maketitle

\begin{abstract}

In this paper we consider the Martin compactification, associated with the operator $\calL = \Delta -1$, of a complete non-compact surface $\Sigma^2$ with negative curvature. In particular, we investigate positive eigenfunctions with eigenvalue one of the Laplace operator $\Delta$ of $\Sigma^2$ and prove a uniqueness result: such  eigenfunctions are unique up to a positive constant multiple if they vanish on the part of the geometric boundary $S_\infty(\Sigma^2)$ of $\Sigma^2$ where the curvature is bounded above by a negative constant, and satisfy some growth estimate on the other part of $S_\infty(\Sigma^2)$ where the curvature approaches zero. This uniqueness result plays an essential role in our recent paper \cite{CaoHesteady3d} in which we prove an infinitesimal rigidity theorem for deformations of certain three-dimensional collapsed gradient steady Ricci soliton with a non-trivial Killing vector field. 
\end{abstract}

\tableofcontents

\section{Introduction}

In this paper, we consider the Martin compactification of a complete non-compact surface $(\Sigma^2, ds^2)$ with  negative curvature. In particular, we investigate positive eigenfunctions with eigenvalue one of the Laplace operator $\Delta$ of $(\Sigma^2, ds^2)$ and prove a certain uniqueness result. The problem arises in our study of deformations of three-dimensional collapsed gradient steady Ricci solitons \cite{CaoHesteady3d}.

A classical result due to G. Herglotz \cite{Herglotz} says that any positive harmonic function $u$ on the unit disk $\mathbb{D}^2$ has an integral representation
\begin{equation*}
u (x) = \int_{\sph^1} K(x, Q)d\sigma(Q), 
\end{equation*}
where $K(x,Q)$ is the Poisson kernel of $\mathbb{D}^2$ and $\sigma$ is a finite positive Borel measure. A fundamental generalization of the above formula to bounded domains in $\Real^n$  was given by R. Martin in \cite{Martin}. He showed that an analogue of the above representation formula holds in complete generality, where the integral is taken over an ideal boundary, called the Martin boundary, defined in terms of the limiting behavior of a Green's function. Later, Martin's results were extended to the second order elliptic differential operators on complete Riemannian manifolds, see, e.g., \cite{Taylor} and the references therein. 

On a Riemannian manifold of non-positive curvature, there is a well-known compactification by attaching its geometric boundary that is determined by the asymptotic behavior of geodesics at infinity, see \cite{EberleinONeill}. A celebrated result by M. Anderson and R. Schoen \cite{AndersonSchoen} states that on a negatively curved manifold $(M^n, g)$, its Martin boundary with respect to the Laplace operator is homeomorphic to its geometric boundary, provided the sectional curvature is pinched between two negative constants. Thus the Martin compactification of such a manifold associated to the Laplace operator coincides with its geometric compactification.  Anderson-Schoen's result has been generalized by A. Ancona in \cite{Ancona} to weakly coercive elliptic operators $L$ using different techniques. In \cite{Ballmann}, W. Ballmann constructed examples of non-positively curved manifolds containing flats of dimension $k\ge 2$ such that the Martin compactification agrees with the geometric compactification.

The Cartan-Hadamard surface $(\Sigma^2, ds^2)$ we are concerned with is diffeomorphic to the upper half-plane $\Real \times (0, \infty)$ with the length element given by
\begin{equation}\label{eqn:ds2surface}
ds^2 = \frac{e^{4y} + 10 e^{2y} + 1}{4(e^{2y}-1)^2} \left(dx^2 + dy^2\right)
\end{equation}
for $(x,y)\in \Real \times (0, \infty)$. We are especially interested in characterizing non-negative functions $W=W(x, y)$ on $\Sigma^2$ such that 
\begin{equation}\label{eqn:PDEW}
\left\{
\begin{array}{rcl}
W_{xx} + W_{yy} -P(y) W & = & 0 \\
& & \\
W(x, 0)  =  0 \quad \& \quad \pd_y W - \half \coth(y) W  & \geq  &0,
\end{array}
\right.
\end{equation}
where
\begin{equation}\label{eqn:P}
P(y) = \frac{e^{4y} + 10e^{2y} + 1} {4(e^{2y}-1)^2}. 
\end{equation}
Note that the Laplace operator of $(\Sigma^2, ds^2)$ is given by $\Delta=P^{-1}(y) (\pd_x^2 + \pd_y^2),$ hence such a function $W$ is in fact an eigenfunction of $\Delta$ with eigenvalue one, 
$$\Delta W= W.$$
This problem of characterizing non-negative solutions of (\ref{eqn:PDEW}) arises in our study of the infinitesimal rigidity of deformations of the collapsed 3D cigar soliton $N^2\times \Real$, the product of Hamilton's cigar soliton $N^2$ and the real line $\Real$ with the product metrics, as follows: consider any one-parameter family of complete  three-dimensional gradient steady Ricci solitons $(M^3(t),g(t), f(t))$ ($0\leq t < \eps$), with $(M^3(0), g(0), f(0))$ being the 3D cigar soliton $N^2 \times \Real$,  satisfying the following two conditions: 
\begin{itemize}
\item the metric $g(t)$ admits  a non-trivial Killing vector field for all $t\in [0, \eps)$,  
\item the scalar curvature $R(t)$ of $g(t)$ attains its maximum at some point on $M^3(t)$ for each $t\in [0, \eps)$.
\end{itemize}
Then, the first variation of the sectional curvatures of $(M^3(t),g(t), f(t))$ produces a certain non-negative function $W=W(x, y)$ on $\Sigma^2$ satisfying (\ref{eqn:PDEW}). The infinitesimal rigidity of the deformation essentially reduces to proving the uniqueness of such nonnegative eigenfunctions up to a positive constant multiple; see \cite{CaoHesteady3d} for the details. This led us to consider the Martin compactification of $(\Sigma^2, ds^2)$ associated with the operator $\calL = \Delta -1$.

It turns out that the Gauss curvature of $(\Sigma^2, ds^2)$ is negative, bounded from below, but approaches zero along some paths to infinity. So we cannot apply the results of Anderson-Schoen \cite{AndersonSchoen} and Ancona \cite{Ancona} in this case. Nevertheless, we shall see that the Martin boundary with respect to the operator $\calL$ is the same as the geometric boundary of $(\Sigma^2, ds^2)$.  Our first results are the following Martin and geometric compactifications of $\Sigma^2$.

\begin{thm}\label{thm:Martinboundary}
The Martin compactification of $\Sigma^2$ associated with the operator $\calL = \Delta -1$ is homeomorphic to the closed half disk $\hat \Sigma = \set{(u,v) \in \Real^2: u^2 + v^2 \leq 1, v\geq 0}$. On the Martin boundary $\pd \hat \Sigma$ we have
\begin{enumerate}
\item the real line $ \set{(x,y) \in \Real^2: y = 0}$ is identified with $\set{ -1<u<1, v=0}$; 
\item the $y$-axis with $y \To \infty$ is identified with the point $(0,1)$;
\item the asymptotic ray $y = x \tan \theta$ is identified with the point $\omega_\theta$ on the semi-circle
\[
\set{\omega_\theta = (\cos\theta,\sin \theta) : \theta \in (0, \pi/2)\cup (\pi/2, \pi)} \subset \pd \hat \Sigma.
\]
\end{enumerate}
\end{thm}
\begin{remark}
\begin{enumerate}[(a)]
\item Each $\omega_\theta$ in case (3) can be approached by geodesic that is asymptotic to the ray $y = x\tan\theta$ as $x\To \infty$. 
\item Under the topology in the Martin compactification, we have 
\[
\lim_{\theta \To 0} \omega_\theta = (1,0) \quad \text{and}\quad \lim_{\theta \To \pi} \omega_\pi = (-1,0).
\]
These two points can be approached by geodesics which are asymptotic to $y = \log\abs{x}$ as $\abs{x}\To \infty$.
\item We show that $\pd\hat \Sigma$ is the minimal Martin boundary and determine the kernel function $K(\cdot, \omega)$, see Proposition \ref{prop:Martinkernel} and Theorem \ref{thm:MartincompSigma}.
\end{enumerate}
\end{remark}

\begin{thm}\label{thm:geomcompactificationintro}
The geometric compactification $\tilde{\Sigma}$ of $\Sigma^2$ with the metric in (\ref{eqn:ds2surface}) is homeomorphic to the Martin compactification $\hat \Sigma$. 
\end{thm}
The detailed description of  $\tilde{\Sigma}$ is given in Theorem \ref{thm:geometriccompactification}.

\begin{remark}
In \cite{CaffarelliLittman}, using an elementary method, L. A. Caffarelli and W. Littman showed that for any positive solution $u$ to the equation $\left(\Delta -1\right)u  = 0 $ on the Euclidean space $\Real^n$ there exists a unique non-negative Borel measure $\mu$ on the unit sphere $\sph^{n-1}$ such that  
\begin{equation*}
u(x) = \int_{\sph^{n-1}} e^{x\cdot \omega}d\mu(\omega).
\end{equation*}
It follows that the (minimal) Martin boundary of $\Real^n$ with respect to $\Delta -1$ is $\sph^{n-1}$.  This is similar to the semi-circle part of $\pd \hat \Sigma$ when $n=2$; note that, on $\Sigma^2$, along the ray $y = x\tan \theta$ the Gauss curvature approaches zero as $\abs{x}\To \infty$. The geometric boundary of $\Real^n$ with the flat metric is also $\sph^{n-1}$, i.e., any point on the Martin boundary can be reached by a geodesic.
\end{remark}

\begin{remark}\label{rem:Laplace}
Since $\left(\Sigma^2,ds^2\right)$ is conformal to the hyperbolic plane, the Martin boundary of $\Sigma^2$ associated with the Laplace operator $\Delta$ is given by the union of the real line $\set{y=0}$ and $\set{\infty}$, the point at infinity, thus is different from the geometric boundary $\pd \tilde{\Sigma}$. Note that the Laplace operator $\Delta$ has zero as the bottom of the $L^2$-spectrum, i.e.,  $\lambda_1(\Sigma^2) =0$ (see Remark \ref{rem:bottomL2spectrum}), and is not weakly coercive.   
\end{remark}

The Martin compactification of a complete Riemannian manifold $M$ with respect to an operator $L$ is also related to the Dirichlet problem at infinity, i.e., given a continuous function $f$ on the geometric boundary $S_\infty(M)$  of $M$, whether there is a unique $L$-harmonic function $w$ on $M$ such that $w =f$ on $S_\infty(M)$. When $L$ is the Laplace operator $\Delta$, the Dirichlet problem at infinity is always solvable if $M$ has negatively pinched curvature $-b^2<K_M<-a^2$ (which was first proved independently by M. Anderson \cite{Anderson} and D. Sullivan \cite{Sullivan}), or if $M$ is one of the examples in \cite{Ballmann}. Note that in both cases, as we mentioned before, there holds the stronger conclusion that Martin and geometric compactifications are homeomorphic.  We remark that the Dirichlet problem at infinity for $\Delta$ on a symmetric space $M$ of noncompact type was investigated by H. F\"{u}rstenberg \cite{Fuerstenberg}; in particular, he showed that the problem can be solved if and only if $M$ has rank one. See also the earlier work of L.-K. Hua \cite{Hua1, Hua2, Hua3} on bounded symmetric domains. For more general Cartan-Hadamard manifolds of rank one (in the sense of \cite{Ballmann1985, Burns}), the solvability of the Dirichlet problem at infinity was proved by Ballmann \cite{BallmannDirichlet}. Moreover, the Poisson integral representation formula was established by Ballmann-Ledrappier \cite{BallmannLedrappier}.  Meanwhile, H. I. Choi \cite{Choi}, Ding-Zhou \cite{DingZhou},  and E. P. Hsu \cite{Hsu} have shown that the Dirichlet problem at infinity for the Laplace operator $\Delta$ is solvable on certain negatively curved manifolds whose curvature approaches zero with certain rate.  Very recently, R. Neel \cite{Neel} has shown that the asymptotic Dirichlet problem for the Laplace operator on a Cartan-Hadamard surface is solvable under the curvature condition $K \le - (1 + \epsilon)/(r^2 \log r)$ (in polar coordinates with respect to a pole) outside of a compact set, for some $\epsilon > 0$.

However, in our case, Theorem \ref{thm:geomcompactificationintro} and Remark \ref{rem:Laplace} imply the following

\begin{corollary}\label{cor:Dirichlet}
The Dirichlet problem at infinity for the Laplace operator $\Delta $ on $\Sigma^2$ is not always solvable. 
\end{corollary}

\begin{remark}  (a) We also characterize those functions $f \in C^0\left(S_\infty(\Sigma^2)\right)$ for which the Dirichlet problem at infinity has a unique solution, see Remark \ref{rem:Dirichletuniqueness}.

(b) Corollary \ref{cor:Dirichlet} gives a new example of a Cartan-Hadamard manifold for which the Dirichlet problem 
at infinity is not solvable in general. Note that the curvature of $\Sigma^2$ is bounded from below by a negative constant, but approaches zero exponentially fast as $y \To \infty$ for each fixed $x$ (see Remark \ref{rem:Gaussdecay}). Previously, when the curvature of a Cartan-Hadamard manifold is assumed to be bounded from above by a negative constant but not from below, Ancona \cite{Ancona2} constructed a counterexample to the solvability of Dirichlet problem at infinity.  See also the work of Borb\'ely \cite{Borbely2}. On the other hand,  there are several papers on the solvability of the Dirichlet problem at infinity when the curvature lower bound has a quadratic growth condition \cite{HsuMarch}, or certain exponential growth conditions \cite{Borbely1, Hsu, Ji}.  
\end{remark}

An eigenfunction $w\in C^\infty(\Sigma)$ of the Laplace operator $\Delta$ with eigenvalue one is also referred as an $\calL$-harmonic function since $\calL w = 0$. Each boundary point $\omega \in \pd\hat \Sigma$ associates a kernel function $K(\cdot, \omega)$ that is positive on $\Sigma$ and $\calL$-harmonic.  The Martin integral representation theorem implies that for any positive $\calL$-harmonic function $w$, there is a (unique) finite non-negative Borel measure $\nu$ on $\pd\hat \Sigma$ such that
\[
w(x,y) = \int_{\pd\hat \Sigma} K(x,y,\omega)d\nu(\omega).
\]
From this integral representation we derive the following uniqueness result of positive eigenfunctions.

\begin{thm}\label{thm:uniquenessWpos}
Suppose that $W$ is a non-negative $\calL$-harmonic function on $\Sigma$ which vanishes on the boundary $\set{y=0}$ and satisfies the following inequality: 
\begin{equation}\label{eqn:Wgrowth}
W(a, y)\geq W(a, b) e^{\frac{y-b}{2}} \quad \text{when} \quad y\geq b
\end{equation}
for some point $(a, b) \in \Sigma^2$. Then either $W=0$, or it is a positive constant multiple of 
\[
W_0(x, y) = \frac{(e^y -1)^2}{e^{\half y}\sqrt{e^{2y}-1}}.
\] 
\end{thm}

\begin{remark}
\begin{enumerate}[(a)]
\item 
Theorem \ref{thm:uniquenessWpos} is used in \cite{CaoHesteady3d} to prove an infinitesimal rigidity result of the gradient steady Ricci soliton $M^3 = N^2 \times \Real$, where $N^2$ is  Hamilton's cigar soliton \cite{Hamiltonsurface}. 

\item 
In \cite{CaoHesteady3d} we used the statement of Theorem \ref{thm:uniquenessWpos} with (\ref{eqn:Wgrowth}) replaced by the  inequality 
\begin{equation*}\label{eqn:Winequality}
\pd_y W -\half \coth(y) W \geq 0,
\end{equation*}
which is a stronger assumption. 
\end{enumerate}
\end{remark}

\smallskip

Now we outline the main steps in the proofs. Since the metric of $\Sigma^2$ is explicit we obtain an integral formula of Green's function $G(x,y)$ of $\calL$ by the classical work of E. Titchmarsh in \cite{Titchmarsh}. The Martin kernel function is derived by asymptotic expansions of $G$ along various paths. In turn, it determines the Martin compactification of $\Sigma$. See Theorem \ref{thm:MartincompSigma} for Theorem \ref{thm:Martinboundary} and Theorem \ref{thm:geometriccompactification} for Theorem \ref{thm:geomcompactificationintro}, while Corollary \ref{cor:Dirichlet} is proved at the end of Section 6, and Theorem  \ref{thm:uniquenessWpos} is proved in Section 7. We refer to the table of contents for an overview of the paper's organization. 

\bigskip{}

\noindent
\textbf{Acknowledgement.}
We are grateful to Werner Ballmann,  Ovidiu Munteanu, Christian Remling, Jiaping Wang, Xiaodong Wang and Meijun Zhu for helpful discussions. Part of the work was carried out while the first author was visiting the University of Macau, where he was partially supported by Science and Technology Development Fund  (Macao S.A.R.) Grant FDCT/016/2013/A1, as well as RDG010 of University of Macau. 

\medskip{}

\section{Preliminaries}

In this section we collect some basics of positive solutions to linear elliptic equations on  complete Riemannian manifolds and Martin compactification of a complete Riemannian manifold $(M^n, g)$ with respect to an elliptic operator $L$. 

\subsection{Martin compactification of complete Riemannian manifolds}

Let $(M^n, g)$ be an $n$-dimensional complete non-compact Riemannian manifold and consider the operator $L = \Delta - 1$. The results in this subsection hold for any second order linear elliptic operator $L$ with uniformly H\"{o}lder continuous coefficients and $L(1) \leq 0$, see for example \cite{Taylor}.

Denote $\Delta M = \set{(x,x): x\in M} \subset M\times M$ the diagonal set. Recall the following 

\begin{definition}\label{defn:Greenfunction}
Let $L$ be a second order linear elliptic operator on $(M^n, g)$. A \emph{Green's function} of $L$ is a function $G: M\times M \backslash \Delta M \To [0, \infty)$ such that the following two equations hold: 
\begin{equation*}
- L_x \int_M G(x,y) \phi(y)dy = \phi(x) 
\end{equation*}
and 
\begin{equation*}
- \int_M G(x,y) L_y \phi(y) dy = \phi(x)
\end{equation*}
for any smooth function $\phi$ with compact support on $M$.
\end{definition}
\begin{remark}
Let $\set{U_n}_{i=1}^\infty$ be an exhaustion of $M$ by relatively compact subsets with $C^2$ boundary, and let $G_{i}$ be the unique Green's function of $L$ on $U_i$ with the Dirichlet boundary condition, then we have
\[
G(x,y) = \lim_{i\To \infty} G_i (x,y) \quad\text{for all } (x,y)\in M\times M \backslash \Delta M
\]
whenever the Green's function $G$ exists, see for example \cite[Proposition 5.6]{Taylor}. Such a Green's function is called minimal, see \cite{LiTam} when $L = \Delta$.
\end{remark}

We assume that $L$ admits a Green's function $G$. In the following we describe the Martin compactification from the limiting behavior of the Green's function $G$, see also \cite{Ballmann}. Fix a reference point $x_0 \in M$ and consider the normalized Green's function
\begin{equation}\label{eqn:Kxy}
K(x, y) =
\left\{
\begin{array}{cl} 
1 & \text{if }  x_0 = x = y \\
& \\
\dfrac{G(x, y)}{G(x_0, y)} & \text{otherwise}.
\end{array}
\right.
\end{equation}
For any fixed $y\in M$, the function $K(\cdot, y)$ is harmonic on $M \backslash \set{y}$ and equals to $1$ at $x_0$. By the Harnack principle, any sequence $\set{x_i}\subset M$ with $\mathrm{dist}(x_0, x_i)\to \infty$ has a subsequence $\set{x_{i_k}}$ such that $\set{K(\cdot, x_{i_k})}$ converges. The limit is a positive $L$-harmonic function on $M$ with value $1$ at $x_0$. Now consider the space of all sequences $\set{x_i}$ in $M$ with $\mathrm{dist}(x_0, x_i)\To \infty$ such that $\set{K(\cdot, x_i)}$ converges. Two such sequences $\set{x_i}$ and $\set{x_i'}$ in $M$ are equivalent if their corresponding limit functions coincide. The \emph{Martin boundary} $\pd_L M$ is defined as the space of equivalence classes. The \emph{Martin topology} on $\hat M^L = M \cup \pd_L M$ induces the given topology on $M$ and is such that a sequence $\set{x_n}\subset \hat M^L$ converges to $\omega \in \pd_L M$ if and only if $\set{K(\cdot, x_i)}$ converges to $K(\cdot, \omega)$. The space $\hat M^L$ is compact with respect to the Martin topology and is called \emph{Martin compactification} of $M$. The Martin topology on $\hat M^L$ is equivalent to the one induced by the following metric:
\[
d(y,z) = \int_M \min\set{1, \abs{K(x,y) - K(x,z)}}f(x) dv_x,
\]
where $f: M \To (0, 1]$ is any positive continuous function and integrable on $M$. Recall 
\begin{definition}
A positive $L$-harmonic function $w$ on $M^n$ is called \emph{minimal} if for any $L$-harmonic function $u\ge 0$, $u \leq w$ on $M$ implies $u = C w$ for some constant $0<C \le 1$. A boundary point $\omega \in \pd_L M$ is called \emph{minimal} if $K(\cdot, \omega)$ is a minimal $L$-harmonic function. Denote $\pd_e M\subset \pd_L M$ the set of all minimal boundary points. 
\end{definition}

Next we collect some basic results on Martin kernels and Martin integral representations, see,  e.g., \cite[Theorem 1.10]{Muratastructure} and \cite[Section 6]{Taylor}.
\begin{thm}\label{thm:Martinintegralrep}
Let $(M^n, g)$ be a complete Riemannian manifold and $L$ an elliptic operator on $M$ with the Green's function $G$. Then, 
\begin{enumerate}[(i)]
\item Any positive minimal $L$-harmonic function is a positive constant multiple of $K(\cdot, \omega)$ for some $\omega \in \pd_e M$.
\item $\pd_e M$ is a $G_\delta$ set, i.e. a countable intersection of open subsets of $\pd_L M$.
\item $K(x,y)$ is continuous on $(M \times \hat M^L) \backslash \Delta M$. 
\item For any positive $L$-harmonic function $u$, there exists a unique finite Borel measure $\nu$ on $\pd_L M$ such that $\nu(\pd_L M \backslash \pd_e M) = 0$ and 
\begin{equation}
u(x) = \int_{\pd_e M} K(x,\omega)d\nu(\omega).
\end{equation}
\end{enumerate}
\end{thm}

\subsection{Positive $L$-harmonic functions and Green's function}
Let $U\subset M$ be a bounded domain. Then, the first eigenvalue of $\Delta$ on $U$ with the Dirichlet boundary condition is given by
\[
\lambda_1(U) = \inf\set{\int_U \abs{\nabla f}^2 dv: \mathrm{supp} f \subset U, \int_U f^2 dv=1}.
\]
Denote $\lambda_1(M)$ the bottom of the $L^2$-spectrum of $\Delta$, then we have
\[
\lambda_1(M) =\lim_{i\To \infty}\lambda_1(U_i)
\]
where $\set{U_i}_{i=1}^\infty$ is any exhaustion of $M$ by relatively compact subsets with $C^2$ boundary. We recall the following well-known result of existence of positive $L$-harmonic function for $L = \Delta + \lambda$, see for example, \cite[Theorem 1]{Fisher-ColbrieSchoen}.
\begin{proposition}
The equation $L u = \Delta u + \lambda u = 0$ for $\lambda \in \Real$ has a positive solution $u$ on $M$ if and only if $\lambda \leq \lambda_1(M)$.
\end{proposition} 

\begin{remark}\label{rem:bottomL2spectrum}
Note that on the surface $\Sigma^2$, the metric defined by (\ref{eqn:ds2surface}) is asymptotic to the flat one as $y$ approaches infinity, so we have $\lambda_1\left(\Sigma^2\right) = 0$.
\end{remark}

For an elliptic operator $L$, while a Green's function always exists locally, the existence of global Green's function requires extra conditions. The result  below follows from \cite[Corollary 5.13]{Taylor}.
\begin{proposition}
A Riemannian manifold $(M^n, g)$ admits a Green's function if either one of the following conditions holds:
\begin{enumerate}
\item The function $1$ is not $L$-harmonic, or
\item there are two non-proportional positive $L$-harmonic functions on $M$.
\end{enumerate}
\end{proposition}

In some special case where $M^n$ is diffeomorphic to $\Real^n$ and $L u = 0$ is a separable equation, Titchmarsh showed that the Green's function of $L$ has an explicit integral form, see \cite{TitchmarshGreen} or \cite[Chapter 15]{Titchmarsh}. 
\begin{thm}\label{thm:TitchmarshGreen}
Let $q(x,y) = q_1(x) + q_2(y)$ be a continuous function on $\Real^2 = \set{(x,y) : x, y\in \Real}$. Denote the Green's functions $G_1(x,\xi, \lambda)$ and $G_2(y,\eta, \lambda)$ of the differential operators $L_1$ and $L_2$ respectively with
\begin{eqnarray*}
L_1 = \frac{d^2}{dx^2} - q_1(x) + \lambda & \text{and} &  L_2 = \frac{d^2}{dy^2} - q_2(y) + \lambda. 
\end{eqnarray*}
Assume that the spectra of $L_1$, $L_2$ are bounded below at $\lambda = \alpha$ and $\lambda = \beta$ respectively. Then for $\Re(\lambda) < \alpha + \beta$, the Green function $G(x,y,\xi,\eta,\lambda)$ associated with the operator 
\[
L = \frac{\pd^2}{\pd x^2} + \frac{\pd^2}{\pd y^2} - q(x,y)+\lambda
\]
has the following integral form
\begin{equation}
G(x, y,\xi, \eta, \lambda) = \frac{1}{2\pi i}\int_{c - i \infty}^{c+i \infty} G_1(x,\xi, \mu)G_2(y,\eta, \lambda - \mu)d\mu
\end{equation}
where the integral is taken along a straight line $\{c+iy\}$ with $\Re(\lambda) - \beta < c < \alpha$. 
\end{thm} 
\begin{remark}
\begin{enumerate}[(a)]
\item The Green's function $G(x,y,\xi,\eta,\lambda)$ is constructed using the exhaustion of $\Real^2$ by rectangles of finite size. It agrees with the construction of the minimal positive Green's function in Definition \ref{defn:Greenfunction} for complete Riemannian manifolds. 
\item Similar formulae hold when there are more than two independent variables. 
\end{enumerate}
\end{remark}

\medskip{}

\section{Geometric compactification of the surface $\Sigma^2$}

In this section we determine all geodesics on $\Sigma^2$ and then the geometric compactification of $\Sigma^2$, see Proposition \ref{prop:Sigmageodesics} and Theorem \ref{thm:geometriccompactification}.

\smallskip

Recall that the surface $\Sigma^2 = \set{(x,y) \in \Real \times (0, \infty)}$ has the length element, see  (\ref{eqn:ds2surface}), of the form
\begin{equation*}
ds^2 = {P(y)}\left(dx^2 + dy^2\right)
\end{equation*}
with
\begin{equation}\label{eqn:Py}
P(y) = \frac{e^{4y} + 10 e^{2y}+1}{4(e^{2y}-1)^2}.
\end{equation}
Clearly, $ds^2$ is a positive definite warped product metric on $\Real\times (0, \infty)$ so it is complete in $x$-direction for any fixed $y$. It is also complete as $y \To \infty$ since $P(y)$ converges to $\frac 1 4$. When $y \To 0$ we have 
\[
\sqrt{P(y)} = \frac{\sqrt 3}{2y} + \frac{y^3}{20\sqrt 3} + O(y^5),
\]
hence it follows that the metric is also complete as $y \To 0$. Therefore, $(\Sigma^2, ds^2)$ is a complete surface. Moreover, its Gauss curvature is given by 

\begin{equation}\label{eqn:GuassK}
K(y) = - \frac{96 e^{2y}\left(e^{8y} + 2 e^{6y} + 18 e^{4y} + 2e^{2y} + 1\right)}{\left(e^{4y} + 10 e^{2y} + 1\right)^3} < 0,
\end{equation}
thus 
\[
\lim_{y \To 0} K(y) = - \frac 4 3 \quad \text{and} \quad \lim_{y \To \infty} K(y) = 0.
\]
Note that the minimal value  $K_{\min} = - \frac 5 3$ of $K(y)$ is achieved at $y = \log(2+\sqrt 3)$. 

To summarize, we have the following 

\begin{proposition}\label{prop:surfacecomplete}
$(\Sigma^2, g)$, with the metric $g$ given by (\ref{eqn:ds2surface}), is a complete surface with negative Gauss curvature bounded below by $-\frac 5 3$. 
\end{proposition}

\begin{remark}\label{rem:Gaussdecay}
Let $r(y)$ denote the distance function to a fixed horizontal line $l$, say $l=\set{(x, 1): x\in \Real}$. Then, by (3.1) and (3.2),  we have the following asymptotic properties: 

$$r(y) \sim \half y \qquad \text{and} \qquad K(y) \sim -96 e^{-4 r(y)}$$ as $y \To \infty$.  In particular, $ K(y)$ approaches zero exponentially fast when $y \To \infty$. 
\end{remark}

Figure \ref{fig:geodesicSigma} shows the typical geodesics on $\Sigma^2$. We sketch the geodesics passing through the $y$-axis at the point $(0,a)$ with $a>0$. The others can be obtained by translation in $x$-direction. 

\begin{itemize}
\item The vertical dashed blue line is of type (i) and it has constant value of $x$. 
\item The red curves are of type (ii) and they have horizontal tangent vector. 
\item The green curves are of type (iii) and they are asymptotic to $y = \log\abs{x}$ for large $\abs{x}$.
\item The purple curves are of type (iv) and they are asymptotic to the rays $y = x\tan\theta$ for large $\abs{x}$ with $0 < \theta < \pi$. 
\end{itemize}

\begin{center}
\begin{figure}[!htp]
\includegraphics[scale=0.9]{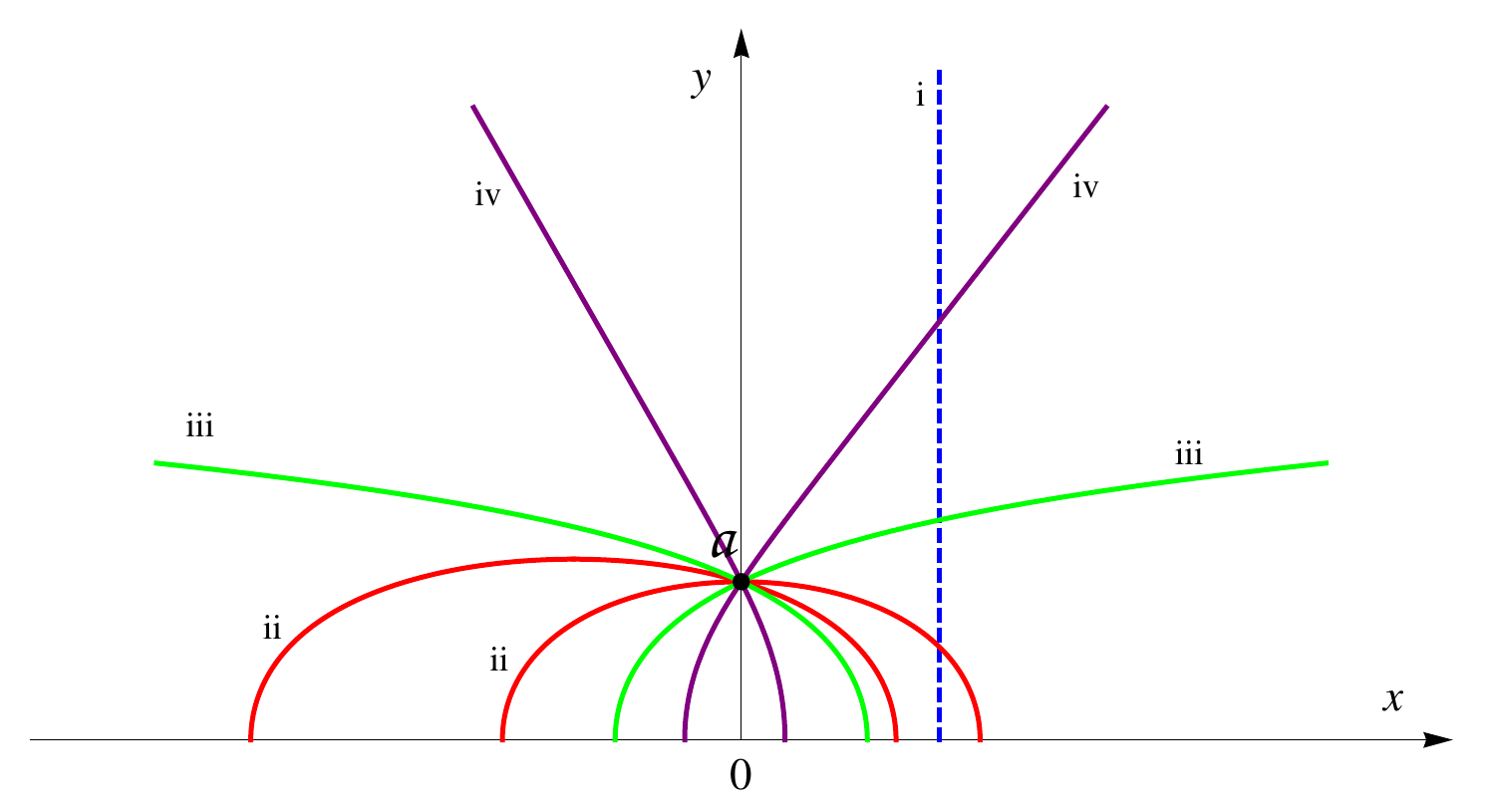}
\caption{Typical geodesics on $\Sigma^2$}\label{fig:geodesicSigma}
\end{figure}
\end{center}

\noindent For the analytic formula of the geodesics of each type, see Proposition \ref{prop:Sigmageodesics} below.

\begin{proposition}\label{prop:Sigmageodesics}
All geodesics of $\Sigma^2$ can be obtained by translation in $x$-direction, the reflection about the $y$-axis or their combinations from the following ones through some point $(0,a)$ with $a>0$:  
\begin{enumerate}[(i)]
\item The $y$-axis.
\item The geodesic has the horizontal tangent vector at $(0,a)$ and it is given by the following formula
\begin{equation}\label{eqn:geodesicxyhorizontal}
x(y) =  \pm \frac{\sqrt{e^{2a}+ 10 + e^{-2a}}}{4\sqrt 3}\tan^{-1}\left(\frac{\sqrt 2 \cosh y \sqrt{\cosh(2a)-\cosh(2y)}}{\cosh(2y) - \sinh^2 a}\right),
\end{equation}
with $y\in (0,a]$.
\item The geodesic has the slope of the tangent vector $m = \frac{\sqrt 3}{\sinh a}$ at $(0,a)$ and it is given by the following formula 
\begin{equation}\label{eqn:geodesicxycosh}
x(y) = \frac{1}{\sqrt 3}\left(\cosh y - \cosh a\right)
\end{equation}
with $y \in (0, \infty)$.
\item The geodesic has the slope of the tangent vector $m > \frac{\sqrt 3}{\sinh a}$ at $(0,a)$ and it is given by the following formula 
\begin{equation}\label{eqn:geodesicxylog}
x(y) = F(y) - F(a)
\end{equation}
with $y \in (0, \infty)$ and 
\begin{eqnarray}
F(y) & = & \frac{\sqrt{3+\sinh^2 a}}{\sqrt{-3+m^2 \sinh^2 a}}\log\Bigg{\{}(-3+m^2 \sinh^2 a)\cosh(y)  \nonumber \\
& & + \sqrt{-3+m^2 \sinh^2 a}\sqrt{(-3+m^2 \sinh^2 a)\cosh^2 y + 3\cosh^2 a + 2m^2 \sinh^2 a}\Bigg{\}}. \label{eqn:geodesicF}
\end{eqnarray}
\end{enumerate}
\end{proposition}

\begin{remark}
The geodesics of type (i), (iii) and (iv) have no horizontal tangent vector at any point. If a geodesic has a horizontal tangent vector somewhere, then it can be obtained by translation in $x$-direction from a geodesic of type (ii). 
\end{remark}

\begin{proof} First of all, the nonzero Christoffel symbols at any point $(x, y)$ are given by
\[
\Gamma^1_{12} = \Gamma^1_{21} = - \Gamma^2_{11} = \Gamma^2_{22} = k(y)
\]
with
\[
k(y) =  -\frac{12 e^{2y}(e^{2y}+1)}{e^{6y} + 9 e^{4y} - 9e^{2y}-1}.
\]
So the geodesic equations are given by
\begin{eqnarray}
x''(t) + 2 k(y) x'(t) y'(t) & = & 0   \label{eqn:geodesicx} \\
y''(t) + k(y)\left(y'(t)^2 - x'(t)^2\right) & = & 0. \label{eqn:geodesicxi}
\end{eqnarray} 
It is obvious that all vertical (half) lines are geodesics. Next,  for any $x_0\in \Real$ and $a>0$ we consider the geodesic $\gamma (t)$ passing through $\gamma(0) = (x_0,a)\in \Sigma^2$ with tangent vector $\gamma'(0)=(1,m)$.  Since the metric is invariant under the translation in $x$-coordinate and the reflection about the $y$-axis, we may assume that $x_0 = 0$ and $m \geq 0$.

From equation (\ref{eqn:geodesicx}) we get 
\begin{equation}\label{eqn:geodesicdxt}
x'(t) = \frac{C_1 (e^{2y}-1)^2}{e^{4y}+10 e^{2y}+1}
\end{equation}
for some constant $C_1$. Since $x'(0) = 1$ and $y(0) = a$ we have
\[
C_1 = \frac{e^{4a}+10 e^{2a}+1}{(e^{2a}-1)^2}.
\]
Let $s(y) = t'(y)^2$, then the second equation (\ref{eqn:geodesicxi}) can be written as
\begin{eqnarray*}
0 & = & \frac{s'(y)}{2s(y)}\left(e^{2y}-1\right)\left(e^{4y}+10 e^{2y}+1\right)^3 - 12 C_1^2 \left(e^{2y}-1\right)^4 \left(e^{4y}+e^{2y}\right)s(y) \\
& & + 12\left(e^{2y}+ e^{4y}\right)\left(1+ 10 e^{2y}+ e^{4y}\right)^2. 
\end{eqnarray*}
It follows that 
\begin{equation}\label{eqn:dsxi}
s(y) = \frac{\left(e^{4y}+10 e^{2y} +1\right)^2}{\left(e^{2y}-1\right)^2\left(C_2\left(e^{4y}+10 e^{2y}+1\right)+12 C_1^2 e^{2y}\right)}
\end{equation}
for some constant $C_2$. 

When $m=0$, since $y'(0)=0$ with $y(0) = a$ the denominator of $s(a)$ vanishes and thus we have
\[
C_2 = - \frac{12e^{2a}\left(e^{4a}+10 e^{2a}+1\right)}{\left(e^{2a}-1\right)^4}.
\]
Equations (\ref{eqn:geodesicdxt}) and (\ref{eqn:dsxi}) imply that 
\begin{equation}\label{eqn:geodesicdxdy}
\frac{dx}{dy} = \pm\frac{\sqrt{e^{2a}+10 + e^{-2a}}}{\sqrt 6} \frac{\sinh(y)}{\sqrt{\cosh(2a)-\cosh(2y)}}.
\end{equation}
Integrating the equation above yields the formula (\ref{eqn:geodesicxyhorizontal}). 

When $m > 0$, the constant $C_2$ can be solved from equation $s(a) = 1/m^2$ as
\begin{eqnarray*}
C_2 = \frac{\left(e^{4a}+10 e^{2a}+1\right)\left(m^2 (e^{2a}-1)^2 -12e^{2a}\right)}{\left(e^{2a}-1\right)^4} = m^2 C_1 - \frac{3C_1}{\sinh^2 a}
\end{eqnarray*}
and we have
\begin{eqnarray}
\frac{dx}{dy} & = &  \frac{\sqrt{10+2\cosh(2a)}\sinh(y)}{\sqrt{(6+5m^2)\cosh(2a)- \left(6+m^2 -m^2 \cosh(2a)\right)\cosh(2y) - 5m^2}} \nonumber \\
& = &  \frac{\sqrt{3+\sinh^2 a}\sinh(y)}{\sqrt{-(3-m^2 \sinh^2 a)\cosh^2 y + 3 \cosh^2 a + 2m^2 \sinh^2 a}}. \label{eqn:dxdym}
\end{eqnarray}

We separate the discussion into three different cases according to the sign of the coefficient $3-m^2 \sinh^2 a$. If $3 - m^2 \sinh^2(a) > 0$, i.e., $m^2 < \frac{3}{\sinh^2 a}$, then equation (\ref{eqn:dxdym}) has the solution $x(y) = H(y) - H(a)$ with $y\in (0, y_0]$ and the function $H$ is given by 
\begin{equation*}
H(y) = \frac{\sqrt{3+\sinh^2 a}}{\sqrt{3-m^2 \sinh^2 a}}\tan^{-1}\left(\frac{\sqrt{3-m^2 \sinh^2 a}\cosh(y)}{\sqrt{-(3-m^2\sinh^2 a)\cosh^2 y + 3 \cosh^2 a + 2m^2 \sinh^2 a }}\right)
\end{equation*}
with 
\[
y_0= \cosh^{-1}\sqrt{\frac{3 \cosh^2 a + 2m^2 \sinh^2 a}{3-m^2 \sinh^2 a}}
\]
Note that the geodesic $\gamma$ has horizontal tangent vector at $(x(y_0), y_0)$ and it can be obtained by translation in $x$-direction $x \mapsto x + x(y_0)$ from the geodesic in case (ii) with $a = y_0$. 

If $3-m^2 \sinh^2(a) = 0$, i.e., $m^2 = \frac{3}{\sinh^2 a}$, then equation (\ref{eqn:dxdym}) has the solution $x(y)$ given in case (iii). 

If $3-m^2 \sinh^2 (a) < 0$, i.e., $m^2 > \frac{3}{\sinh^2 a}$, then equation (\ref{eqn:dxdym}) has the solution as in case (iv). This finishes the proof.
\end{proof}

\begin{thm}\label{thm:geometriccompactification}
The geometric compactification of $\Sigma^2$ is homeomorphic to the half disk $\tilde{\Sigma} = \set{(u,v) \in \Real^2 : u^2 + v^2\leq 1, v\geq 0 }$. On the geometric boundary $S_\infty(\Sigma) = \pd \tilde{\Sigma}$ we have
\begin{enumerate}
\item the geodesics that approach to points of the real line $\set{(x,y) \in \Real^2 : y=0}$ are identified with points of the interval $\set{-1 < u< 1, v=0}$,
\item the vertical half lines with $y \To \infty$ are identified with the point $(0, 1)$,
\item the geodesics asymptotic to $x =\pm \frac{1}{\sqrt 3}\left(\cosh y - \cosh a\right)$ for some $a>0$ are identified with the point $(\pm 1,0)$, 
\item the geodesics asymptotic to $x = \pm\left(F(y) - F(a)\right)$ with $F$ given in equation (\ref{eqn:geodesicF}) for some $a>0$ are identified with the point $\omega_\theta = (\cos\theta, \sin\theta)$($\theta \in (0, \pi/2) \cup (\pi /2, \pi)$) with 
\begin{equation}\label{eqn:tanthetam}
\tan \theta = \pm \frac{\sqrt{-3 + m^2 \sinh^2 a}}{\sqrt{3+\sinh^2 a}}
\end{equation}
and $m > \frac{\sqrt 3}{\sinh a}$.
\end{enumerate}
\end{thm}
\begin{proof}
Without loss of generality we consider the geodesic $\gamma(t)$ starting from the point $(0,1)\in \Sigma$, i.e., $\gamma(0) = (0,1)$. We choose the parameter $t$ such that $\abs{\gamma'(0) } = 1$. If $\gamma'(0) = (0,1)$, then $\gamma(t)$ with $t>0$ is the $y$-axis with $y>1$. If $\gamma'(0) = (0,-1)$, then it approaches to $(0,0) \in \Real^2$ along the $y$-axis with $0<y<1$. Next we assume that $\gamma'(0)$ is not parallel to the $y$-axis. Denote by $\phi$ the angel from the positive $x$-axis to $\gamma'(0)$ and $m = \tan \phi$. We consider the case when $\phi \in (-\pi/2, \pi/2)$. The argument for $\phi\in (\pi/2, 3\pi/2)$ is similar. 

From the proof of Proposition \ref{prop:Sigmageodesics}, we see that when $m < \frac{\sqrt 3}{\sinh(1)}$, $\gamma(t)$ approaches to the positive $x$-axis as $t\To \infty$ and the limits $\lim_{t\To \infty}\gamma(t)$ cover the whole positive $x$-axis $(0, \infty)$. When $m = \frac{\sqrt{3}}{\sinh(1)}$, the geodesic is given by $x = \frac{1}{\sqrt 3}\left(\cosh y - \cosh 1\right)$ and it is asymptotic to the curve $y = \log x$ for large $x>0$. When $m > \frac{\sqrt 3}{\sinh(1)}$, the geodesic is given by $x = F(y) - F(1)$ as $F$ in equation (\ref{eqn:geodesicF}). It is asymptotic to the ray $y = \frac{\sqrt{-3 + m^2 \sinh^2(1)}}{\sqrt{3+ \sinh^2(1)}} x$ for large $x>0$.  So we have the homeomorphism from the directions at $(0,1) \in \Sigma$ to $S_\infty(\Sigma)$ such that $\phi = -\pi/2$ is identified with $(0,0)$, $\phi = \pi/2$ with $(1,0)$, $m \in (-\infty, \sqrt{3}/\sinh(1))$ with $\set{(u, 0): 0<u <1}$, $m = \sqrt{3}/\sinh(1)$ with $(0,1)$, and $m \in (\sqrt 3/\sinh(1), \infty)$ with $(\cos\theta, \sin\theta)$($0< \theta < \pi/2$) by equation (\ref{eqn:tanthetam}) with the plus sign. 
\end{proof}

\medskip{}

\section{The minimal Green's function of the operator $\calL = \Delta -1$}

In this section we prove an integral formula of the Green's function $G(x, y, \xi, \eta)$  of 
$\calL = \Delta - 1$, see Theorem \ref{thm:Greenintegral}. Recall that the Laplace operator of $(\Sigma^2, ds^2)$ is given by $$\Delta=P^{-1}(y) (\pd_x^2 + \pd_y^2),$$ so we have $$\calL = \Delta - 1= P^{-1}(y) (\pd_x^2 + \pd_y^2 - P(y)).$$
Here and in Section 5 we shall use various properties of certain special functions, e.g., the Gamma function, the Gauss hypergeometric function, etc. We refer the reader to \cite{DLMF} and \cite{OLBC} for more details.  

\smallskip

First of all, we define a few relevant functions for the rest of the paper. These functions arise in the study of the spectral properties of the differential operator $A = -D^2_x + P(x)$ defined on $(0, \infty)$, see Appendix \ref{sec:STspectrum}. For any complex number $\lambda \in\Cpx$, we write 
\[
\frac 1 4 - \lambda = re^{i\phi}\quad \text{ with } r\geq 0 \text{ and }\phi \in [-\pi, \pi)
\]
and define the function 
\begin{equation}\label{eqn:alphaphi}
\alpha(\lambda) = \sqrt{\frac 1 4 - \lambda} = \sqrt r e^{i\frac \phi 2}.
\end{equation}
If we use the phase angle $\theta \in [0, 2\pi)$ instead and  write
\[
\lambda - \frac 1 4 = r e^{i \theta}\quad \text{ with } r> 0 \text{ and } \theta \in[0, 2\pi),
\] 
then we have
\begin{equation}\label{eqn:alphatheta}
\alpha(\lambda) =  - i \sqrt{r}e^{i\frac{\theta}{2}}.
\end{equation}
The function $\alpha(\lambda)$ has branch cut along $[\frac 1 4, \infty)$. Using $\alpha = \alpha(\lambda)$ we define the following functions for $\lambda \in \Cpx$ and $y> 0$:
\begin{equation}
a(\lambda) = - \frac{\Gamma(1+\alpha)\Gamma\left(\frac 3 2 - \alpha\right)}{\Gamma(1-\alpha)\Gamma\left(\frac 3 2 + \alpha\right)},
\end{equation}
\begin{eqnarray}
w_1(\lambda, y) & = & e^{(1-\alpha)y}\left(e^{2y}-1\right)^{-\half}\hF{-\half, -\half +\alpha, 1+\alpha; e^{-2y}},\\
w_2(\lambda, y) & = & e^{(1+\alpha)y}\left(e^{2y}-1\right)^{-\half}\hF{-\half, -\half -\alpha, 1-\alpha; e^{-2y}}.
\end{eqnarray}
Here $\hF{a,b,c; z} =\ _2 F_1(a,b,c;z)$ is the Gauss hypergeometric function. We also denote 
\begin{eqnarray}
w_a(\lambda, y) & = & w_1 + a(\lambda) w_2, \\
w_b(\lambda, y) & = & w_1.
\end{eqnarray}
The Wronskian of $w_a$ and $w_b$ is computed in the proof of Theorem A.2 and is given by
\begin{equation}
W(w_a,w_b) = -2\alpha a(\lambda).
\end{equation}

We define the following differential operators: 
\begin{eqnarray}
L_1 & = & \frac{d^2}{dx^2} - \lambda, \\
L_2 & = & \frac{d^2}{dy^2} - P(y) + \lambda. 
\end{eqnarray}

\begin{lemma}
The Green functions $G_i$ of $L_i$ ($i=1,2$) are given by
\begin{eqnarray}
G_1(x, \xi, - \lambda) & = & \frac{e^{-\abs{x-\xi}\sqrt{\lambda}}}{2\sqrt{\lambda}}, \\
G_2(y, \eta, \lambda) & = & \frac{w_1(\lambda, y)w_1(\lambda, \eta)}{2\alpha a(\lambda)} + \frac{w_2(\lambda, y)w_1(\lambda, \eta)}{2\alpha} \label{eqn:G2w1w2} 
\end{eqnarray}
for $0 < y \leq \eta < \infty$. Here $\sqrt{\lambda} = \sqrt{\rho} e^{i \frac{\phi}{2}}$ if $\lambda = \rho e^{i\phi}$ with $\rho\geq 0$ and $\phi \in [-\pi, \pi)$. \\
Moreover the Green function $G_1(x,\xi, -\lambda)$ has the branch cut along $\lambda \in (-\infty, 0]$. 
\end{lemma}

\begin{proof}
The Green function $G_1$ of $L_1$ is well-known. We now show the formula of $G_2$. From the proof of Theorem \ref{thm:expansionS}, we know that $L_2\left(w_a(\lambda , y)\right) = L_2\left(w_b(\lambda, y)\right) = 0$, $w_a$ is square integrable near $y = 0$, and $w_b$ is square integrable near $y = \infty$. Thus, we have  
\[
G_2(y,\eta, \lambda) = \frac{w_a(\lambda, y)w_b(\lambda, \eta)}{W(w_b, w_a)}
\]
and it gives the desired formula after the substitution by $w_1$ and $w_2$.
\end{proof}

In the following we determine the singularities of $G_2(y,\eta, \lambda)$ for fixed $y, \eta \in (0, \infty)$. First we prove an identity of hypergeometric functions in $w_1$ and $w_2$. 

\begin{lemma}\label{lem:connectionformula}
We have the formula
\begin{eqnarray*}
& & \frac{e^{zy}}{\Gamma(1-z)\Gamma(3/2+ z)} \hF{-\half, -\half-z, 1-z; e^{-2y}} \\
& &  - \ \frac{e^{-zy}}{\Gamma(1+z)\Gamma(3/2- z)} \hF{-\half, -\half+z, 1+z; e^{-2y}} \\
& = & \frac{\sqrt{1-e^{-2y}}}{\sqrt{2\pi}}\left(\cosh y\right)^{-z} \tanh^{\frac 3 2}(y) {\sin(\pi z)} \hF{\half z+ \frac 5 4, \half z + \frac 3 4, 2;\tanh^2 y}
\end{eqnarray*}
for any $z\in \Cpx$ and $y \in (0, \infty)$. Here the function $(\cosh y)^{-z}$ is defined by
\[
\left(\cosh y\right)^{-z} = e^{-z\log \cosh y}.
\]
\end{lemma}

\begin{proof}
It follows from the connection formula \cite[14.9.15]{DLMF} of Legendre functions that
\[
\frac{P^\mu_\nu(x)}{\Gamma(\nu+\mu +1)} - \frac{P^{-\mu}_\nu(x)}{\Gamma(\nu - \mu +1)} = 2\frac{\sin \mu \pi}{\pi} e^{-\mu \pi i}\frac{Q^{\mu}_{\nu}(x)}{\Gamma(\nu+\mu +1)},
\]
where $\mu, \nu \in \Cpx$ and $x\in \Real$. From \cite[14.3.6-14.3.7]{DLMF}, when $x\in (1, \infty)$ the Legendre functions can be represented by the hypergeometric functions as
\begin{eqnarray*}
P^{\mu}_{\nu}(x) & = & \left(\frac{x+1}{x-1}\right)^{\frac{\mu}{2}}\frac{1}{\Gamma(1-\mu)}\hF{\nu+1, -\nu, 1-\mu; \half - \half x}, \\
Q^{\mu}_{\nu}(x) & = & e^{\mu \pi i}\frac{\sqrt{\pi} \Gamma(\nu+\mu+1)(x^2-1)^{\frac \mu 2}}{2^{\nu+1}x^{\nu+\mu+1}\Gamma(\nu+3/2)} \hF{\half\nu + \half \mu +1, \half \nu+\half\mu+\half, \nu+\frac 3 2; \frac{1}{x^2}}.
\end{eqnarray*}
In the equation of $Q^{\mu}_{\nu}(x)$, it is assumed that $\mu+\nu\ne -1, -2, \ldots$. Applying the connection formula with $\mu = z$, $\nu = \half$ and $x={\coth  y}$, we obtain
\begin{eqnarray*}
& & \frac{e^{z y}}{\Gamma(1-z)\Gamma(3/2+z)}\hF{\frac 3 2, -\half, 1-z; \frac{1}{1-e^{2 y}}} \\
& &  - \ \frac{e^{-z y}}{\Gamma(1+z)\Gamma(3/2-z)}\hF{\frac 3 2, -\half, 1+z; \frac{1}{1-e^{2 y}}} \\
& = &\frac{\sin\pi z}{\sqrt{2\pi}} \frac{\tanh^{\frac 3 2}( y)}{e^{z\cosh y}} \hF{\half z+\frac 5 4, \half z + \frac 3 4,2; \tanh^2 y}.
\end{eqnarray*}
The transformation formula
\[
\hF{a,b,c; w} = (1-w)^{-a} \hF{a, c-b, c; \frac{w}{w-1}}
\]
implies that 
\[
\hF{-\half, -\half \pm z, 1\pm z; e^{-2 y}} = {\left(1-e^{-2 y}\right)^\half}\hF{-\half, \frac 3 2, 1\pm z; \frac{1}{1-e^{2 y}}}
\]
which gives us the desired identity. When $z+\half = -1, -2, \ldots$, the first term on the left hand side vanishes as $\Gamma(3/2 + z)$ goes to infinity and the identity still holds where the hypergeometric functions reduces to polynomials. 
\end{proof}

\begin{lemma}
For any fixed $ y, \eta \in (0, \infty)$, the Green function $G_2( y, \eta, \lambda)$ has the branch cut along $[\frac 1 4, \infty)$ and is analytic on the principal branch.
\end{lemma}
\begin{proof}
Since $G_2$ is the Green function of the differential operator $L_2$, it is analytic for non-real $\lambda\in \Cpx\backslash \Real$. We discuss the singularities of $G_2$ on the real axis. 

From the formulas of $w_1, w_2$ and $a(\lambda)$ given in (4.3)-(4.5), it has branch cut along $\lambda \in [\frac 1 4, \infty)$. By Lemma \ref{lem:connectionformula}, $G_2$ can be written as
\begin{eqnarray*}
G_2( y, \eta, \lambda) &=& \frac{e^{(1-\alpha)\eta}}{2\alpha \sqrt{(e^{2 y}-1)(e^{2\eta}-1)}}\left\{\frac{1}
{a(\lambda)} \hF{-\half, -\half + \alpha, 1+\alpha; e^{-2 y}} e^{(1-\alpha) y} \right. \\
& & \left. + e^{(1+\alpha) y}\hF{-\half, -\half-\alpha, 1-\alpha; e^{-2 y}}\right\}\hF{-\half, -\half+\alpha, 1+ \alpha; e^{-2\eta}} \\
& = & \frac{e^{-\alpha \eta}}{2\alpha\sqrt{2\pi}\sqrt{1-e^{-2\eta}}} \left(\cosh  y\right)^{-\alpha} \tanh^{\frac 3 2}( y) \hF{\half \alpha + \frac 5 4, \half \alpha + \frac 3 4,2; \tanh^2 y} \times\\
& & \sin(\pi \alpha)\Gamma(1-\alpha)\Gamma(3/2+\alpha) \hF{-\half, -\half+\alpha, 1+\alpha; e^{-2\eta}}. 
\end{eqnarray*}
Since $\sin (\pi \alpha) =0$ when $\alpha$ is an integer and the Gamma and hypergeometric functions have only simple poles, the possible poles of $G_2$ are given by $3/2 + \alpha = -n$ for $n=0, 1,2,\ldots$, i.e., $\alpha = -(n + 3/2)$. Note that $i \alpha = \sqrt r e^{i\frac \theta 2}$ with $\frac{\theta}{2} \in [0, \pi)$, this case does not exist and so $G_2$ is analytic in $\lambda$.
\end{proof}

Note that the Green function of $\calL= \Delta - 1$ is the same as the one of $L = \pd_x^2 +\pd_y^2 - P(y)$. From Theorem \ref{thm:TitchmarshGreen}, for any real constant $c\in (0, \frac 1 4)$, the Green function of the operator $L = \pd_x^2 + \pd_y^2 - P(y) $ is the given by
\begin{equation}\label{eqn:Greenconvolution}
G(x, y, \xi, \eta) = \frac{1}{2\pi i}\int_{c- i \infty}^{c + i \infty} G_1(x, \xi, -\lambda) G_2(y, \eta, \lambda) d\lambda.
\end{equation}

\begin{thm}\label{thm:Greenintegral}
The Green function $G(x, y, \xi, \eta)$ of the operator $\calL = \Delta -1$ and $L= \pd_x^2 + \pd_y^2 - P(y)$ has the following integral form: 

\begin{equation}\label{eqn:GreenRe}
G(x, y, \xi, \eta) = \frac{1}{2\pi}\int_0^\infty \Re\left\{a(s) w(s, y) w(s, \eta) + \bar w(s, y) w(s, \eta)\right\}\frac{e^{-\abs{x-\xi}\sqrt{s^2+\frac 1 4}}}{\sqrt{s^2 + \frac 1 4}}ds
\end{equation}
or 
\begin{equation}\label{eqn:GreenR}
G(x, y, \xi, \eta) = \frac{1}{\pi}\int_{-\infty}^\infty \left\{a(s) w(s, y) w(s, \eta) + \bar w(s, y) w(s, \eta)\right\}\frac{e^{-\abs{x-\xi}\sqrt{s^2+\frac 1 4}}}{\sqrt{s^2 + \frac 1 4}}ds
\end{equation}
for $y, \eta \in (0, \infty)$ and $x, \xi \in \Real$ with $(x, y)\ne (\xi, \eta)$, and the integral in (\ref{eqn:GreenR}) takes the principal value. Here, for $s \geq 0$,  
\begin{equation}\label{eqn:as}
a(s)=: a(1/4 +s^2) = - \frac{\Gamma(1- is)\Gamma(3/2+ is)}{\Gamma(1+ i s)\Gamma(3/2 - i s)}
\end{equation}
and
\begin{equation}\label{eqn:wsy}
w(s,y) = \frac{e^{(1-is) y}}{\sqrt{e^{2 y}-1}}\hF{-\half, -\half + is, 1+i s; e^{-2 y}}
\end{equation}
for $y \in (0, \infty)$. 
\end{thm}

We need a lemma that will be used in the proof of Theorem \ref{thm:Greenintegral}.

\begin{lemma}\label{lem:asymptoticGammaF}
We have the following asymptotic expansions:  
\begin{eqnarray*}
\frac{\Gamma(1+z)\Gamma(3/2 - z)}{\Gamma(1-z)\Gamma(3/2+ z)} & \sim &  \frac{1-4z^2}{4z^2} \left(1 - \frac{1}{8}z^{-1} + O(z^{-2})\right)^2 \tan(\pi z) \quad (\abs{\arg z} < \pi), \\
\hF{-\half, -\half + z, 1+z; e^{-2y}} & \sim & \left(1- e^{-2y}\right)^{\half}\left(1- \frac{3}{4(1+z)}\frac{1}{1-e^{2y}} + O(z^{-2})\right)
\end{eqnarray*}
when $\abs{z} \To \infty$. Here $\arg z \in [-\pi,\pi).$
\end{lemma}
\begin{proof}
We show the first formula which involves gamma functions. Since 
\[
\Gamma(1+z) = z \Gamma(z) 
\]
and 
\[
\Gamma(z)\Gamma(-z) = - \frac{\pi}{z \sin(\pi z)}, 
\]
we have
\begin{eqnarray*}
\Gamma(1-z) & = & - z \Gamma(-z) = \frac{\pi}{\sin(\pi z) \Gamma(z)} \\
& = & \frac{\pi z}{\sin(\pi z)}\frac{1}{\Gamma(1+z)}.
\end{eqnarray*}
Since
\begin{eqnarray*}
\Gamma\left(\frac 3 2 + z\right) \Gamma\left(\frac 3 2 - z\right) & = & \Gamma(3/2)^2\frac{1}{\cos(\pi z)}\left(1-4z^2\right) \\
& = & \frac{\pi(1-4z^2)}{4 \cos(\pi z)}
\end{eqnarray*}
we have
\[
\Gamma(3/2-z) = \frac{\pi (1-4z^2)}{4\cos(\pi z)}\frac{1}{\Gamma(3/2+z)}.
\]
It follows that when $\abs{z} \To \infty$ we have
\begin{eqnarray*}
\frac{\Gamma(1+z)\Gamma(3/2- z)}{\Gamma(1-z)\Gamma(3/2+ z)} & = & \frac{\Gamma(1+z)^2}{\Gamma(3/2+ z)^2}\frac{\sin(\pi z)}{\pi z} \frac{\pi (1-4z^2)}{4\cos(\pi z)} \\
& = & \frac{(1-4z^2)}{4z}\frac{\sin(\pi z)}{\cos(\pi z)} \left(\frac{\Gamma(1+z)}{\Gamma(3/2+z)}\right)^2 \\
& \sim & \frac{(1-4z^2)}{4z^2}\frac{\sin(\pi z)}{\cos(\pi z)} \left(1- \frac 3 8 z^{-1} + O(z^{-2})\right)^2.
\end{eqnarray*}
In the last step we have used the following expansion for large $\abs{z}$ with $\abs{\arg z}< \pi$:
\[
\frac{\Gamma(z+\alpha)}{\Gamma(z+\beta)} \sim z^{\alpha - \beta}\left(1+ \frac{(\alpha -\beta)(\alpha+\beta -1)}{2} z^{-1} + O(z^{-2})\right).
\]

Next we consider the hypergeometric function. From the transformation formulas we have 
\begin{eqnarray*}
\hF{-\half,-\half + z, 1+z; e^{-2 y}} &= & \left(1- e^{-2 y}\right)^{\half}\hF{-\half, \frac 3 2, 1+z; \frac{e^{-2 y}}{e^{-2 y}-1}} \\
& = & \left(1- e^{-2 y}\right)^{\half}\hF{-\half, \frac 3 2, 1+z; \frac{1}{1-e^{2 y}}}.
\end{eqnarray*}
Write $z = u +i v$ with $u, v \in \Real$. Then we have 
\[
\frac{1}{1-e^{2 y}} < 0 < \half \quad (y>0)
\]
and 
\[
\abs{1+z+n} =  \sqrt{(n+1+u)^2 + v^2}\geq 1
\]
for all $n=0, 1, 2, \ldots$. So from \cite[Section 15.12]{DLMF} we have the following asymptotic expansion
\[
\hF{-\half, \frac 3 2, 1+z; \frac{1}{1-e^{2 y}}} \sim \sum_{s=0}^{m-1}\frac{(-\half)_s (\frac 3 2)_s}{(1+z)_s s!} \frac{1}{(1-e^{2 y})^s} + O(z^{-m})
\]
for any fixed $m = 1, 2, \ldots$. Letting $m=2$ in the formula above gives us the desired expansion.
\end{proof}

\begin{proof}[Proof of Theorem \ref{thm:Greenintegral}]
First we show that formulas (\ref{eqn:GreenRe}) and (\ref{eqn:GreenR}) are equivalent. Note that $\bar a(s) = a(-s)$ and $\bar w(s, y) = w(-s, y)$. It follows that the real part of 
\[
a(s) w(s,y)w(s,\eta) + \bar w(s,y) w(s,\eta)
\] 
is an even function in $s$ and its imaginary part is odd in $s$. So the two integral formulas are equal. 

In the following we show formula (\ref{eqn:GreenRe}). We first assume that $y \leq \eta$. 
Consider the simple closed contour $\Omega$ in Figure \ref{fig:contourG}.
\begin{center}
\begin{figure}[!htp]
\includegraphics[scale=0.6]{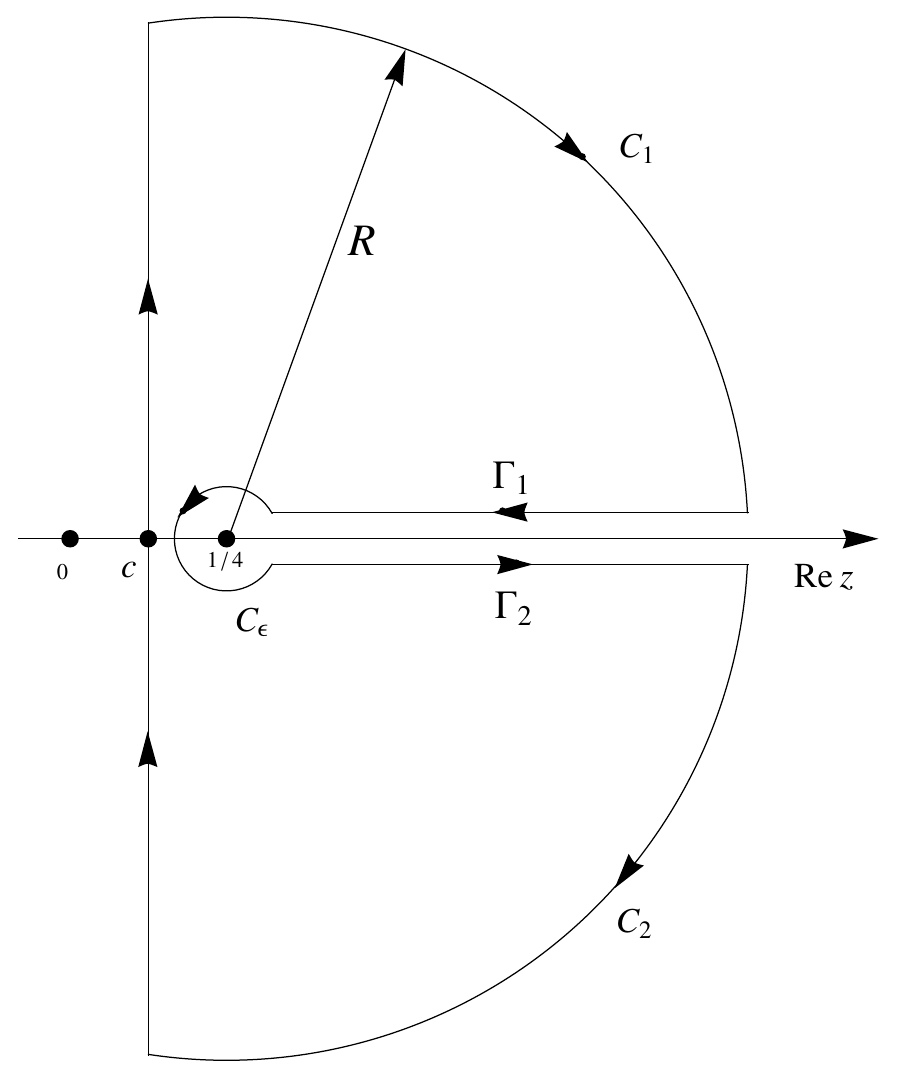}
\caption{Contour $\Omega$ used for the Green function $G(x, y, \xi, \eta)$}\label{fig:contourG}
\end{figure}
\end{center}
It consists of the circular arcs $C_1$ and $C_2$ with radius $R$ centered at $\frac 1 4$, the vertical line from $c - i \infty$ to $c + i \infty$ cut by the arcs $C_1$ and $C_2$, the small circle $C_\eps$ with radius $\epsilon$ centered at $\frac 1 4$, and the horizontal line segments $\Gamma_1$, $\Gamma_2$. Since the Green functions $G_1$ and $G_2$ have no singularities in the region bounded by $\Omega$, the Cauchy Theorem implies that  
\begin{equation}\label{eqn:Cauchyidentity}
\oint_\Omega  G_1(x,\xi,-\lambda) G_2(y, \eta, \lambda) d\lambda = 0.
\end{equation}

\begin{claim}
We have
\begin{equation*}
\lim_{\eps\To 0}\int_{C_\eps} G_1(x,\xi,-\lambda) G_2(y, \eta, \lambda) d\lambda = 0.
\end{equation*}
\end{claim}
On the circle $C_\eps$ we have 
\[
\lambda = \frac 1 4 + \eps e^{i\theta} = r e^{i\phi}
\] 
with $\theta \in [0, 2\pi)$ and $\phi \in [-\pi, \pi)$. As $\eps \To 0$ we have $r\approx \frac 1 4$ and $\phi \approx 0$. It follows that 
\begin{eqnarray*}
G_1(x,\xi, -\lambda) = \frac{e^{-\abs{x-\xi}\sqrt{\lambda}}}{2\sqrt{\lambda}} = \frac{1}{2\sqrt{r}e^{i\frac{\phi}{2}}}\exp\left(- \abs{x-\xi}\sqrt{r}\left(\cos \frac{\phi}{2} + i \sin\frac{\phi}{2}\right)\right)
\end{eqnarray*}
and then 
\begin{eqnarray*}
\abs{G_1(x,\xi, -\lambda)} = \frac{1}{2\sqrt r}\exp\left(-\abs{x-\xi}\sqrt{r}\cos\frac{\phi}{2}\right).
\end{eqnarray*}
So we have $\abs{G_1(x,\xi,-\lambda)} \leq c_1$ for some constant $c_1$ as $\eps \To 0$. For the Green function $G_2$, we have $\alpha = - i \sqrt{\eps} e^{i\frac{\theta}{2}}$ that converges to $0$ as $\eps\To 0$. It follows that when $\eps \To 0$ we have the following limits
\begin{eqnarray*}
a(\lambda) & \To & 1 \\
w_1(\lambda, y) & \To & \frac{e^{y}}{\sqrt{e^{2y}-1}}\hF{-\half, -\half, 1; e^{-2 y}} \\
w_2(\lambda, y) & \To & \frac{e^{y}}{\sqrt{e^{2y}-1}}\hF{-\half, -\half, 1; e^{-2 y}}
\end{eqnarray*}
and then
\[
\abs{G_2(y, \eta, \lambda)} \leq \frac{c_2}{\abs{\alpha}} = \frac{c_2}{\sqrt{\eps}}
\]
where the constant $c_2$ does not depend on $\eps$. So we have
\[
\abs{G_1(x,\xi, -\lambda)G_2(y,\eta, \lambda ) d\lambda} \leq \frac{c_1 c_2}{\sqrt{\eps}} \eps d\theta = c_1 c_2 \sqrt{\eps} d\theta
\]
and the integral along $C_\eps$ converges to zero as $\eps \To 0$. This finishes the proof of the claim on $C_\eps$.

\begin{claim}
We have
\begin{equation*}
\lim_{R \To \infty}\int_{C_1 \cup C_2} G_1(x, \xi, -\lambda) G_2(y, \eta, \lambda) d\lambda = 0.
\end{equation*}
\end{claim}

On the arcs $C_1$ and $C_2$ we have $\lambda = \frac 1 4 + R e^{i\theta}$ and $\theta \in [\theta_1, \theta_2]$. It follows that 
\[
\alpha = - i \sqrt R e^{i\frac{\theta}{2}}.
\]
In the following we consider the asymptotic expansions of various functions as $R \To \infty$. For the fundamental solutions $w_i$'s, from Lemma \ref{lem:asymptoticGammaF} we have
\begin{eqnarray*}
w_1(\lambda, y) & = & e^{(1-\alpha)y}(e^{2 y}-1)^{-\frac{1}{2}}\hF{-\half, - \half + \alpha, 1+ \alpha; e^{-2 y}} \\
& \sim & e^{-\alpha  y}\left(1 - \frac{3}{4(1+\alpha)}\frac{1}{1-e^{2 y}} + O(\alpha^{-2})\right), \\
w_2(\lambda,  y) & \sim & e^{\alpha  y}\left(1 - \frac{3}{4(1-\alpha)}\frac{1}{1-e^{2 y}} + O(\alpha^{-2})\right).
\end{eqnarray*}
So we have
\begin{eqnarray*}
2\abs{\alpha G_2( y, \eta, \lambda)} & = & \abs{\frac{w_1(\lambda,  y)w_1(\lambda, \eta)}{a(\lambda)} + w_2(\lambda,  y)w_1(\lambda, \eta)} \\
& \leq & \abs{\frac{w_1(\lambda, y)w_1(\lambda, \eta)}{a(\lambda)}} + \abs{w_2(\lambda,  y)w_1(\lambda, \eta)} \\
& \leq & c_1 \abs{e^{-\alpha( y + \eta)}}\abs{\frac{\sin(\pi \alpha)}{\cos(\pi \alpha)}} + c_2 \abs{e^{-\alpha(\eta -  y)}}
\end{eqnarray*}
for some constants $c_1, c_2 >0$. For the exponential functions we have
\begin{eqnarray*}
\abs{e^{-\alpha( y + \eta)}} & = & \abs{\exp\left(i ( y + \eta)\sqrt R\left(\cos \frac\theta 2 + i \sin \frac\theta 2\right)\right)} \\
& = & \exp\left(-( y + \eta)\sqrt R \sin \frac\theta 2\right)
\end{eqnarray*}
and
\begin{equation*}
\abs{e^{-\alpha(\eta - y)}} = \exp\left(-(\eta -  y)\sqrt R \sin \frac\theta 2\right).
\end{equation*}
We also have
\begin{eqnarray*}
\abs{\frac{\sin \pi \alpha}{\cos \pi \alpha}} = \abs{\frac{\sinh(i \pi \alpha)}{\cosh( i \pi \alpha)}} = \abs{\frac{\sinh(\pi \sqrt R e^{i\frac \theta 2})}{\cosh(\pi \sqrt R e^{i \frac\theta 2})}}. 
\end{eqnarray*}
The real part is given by
\[
\Re (\pi \sqrt R e^{i\frac \theta 2}) = \pi \sqrt R \cos \frac \theta 2.
\]
On the arc $C_1$ since $\theta_1 \approx 0$ and $\theta_2 \approx \frac{\pi}{2}$ we have $0 \lesssim \frac \theta 2 \lesssim \frac\pi 4$. On the arc $C_2$ we have $\frac{3\pi}{4}\lesssim \frac{\theta}{2} \lesssim \pi$. On both arcs we have $\abs{\cos \frac \theta 2} > c_0$ for some constant $c_0 > 0$ and so the real part of $\pi \sqrt R e^{i\frac \theta 2}$ diverges to infinite and we have the following limit
\[
\abs{\frac{\sin\pi \alpha}{\cos\pi\alpha}} \To 1 \quad \text{ as } \quad R \To \infty.
\]
It follows that for large $R > 0$ we have
\[
\abs{\frac{\sin \pi \alpha}{\cos \pi \alpha}} < c_3
\]
for some constant $c_3 > 1$. Combining these estimates, we have
\[
\abs{G_2( y, \eta, \lambda)} \leq \frac{c_1}{\sqrt R} \exp\left(-\sqrt R ( y + \eta)\sin\frac\theta 2 \right) + \frac{c_2}{\sqrt R} \exp\left(-\sqrt R (\eta -  y)\sin\frac\theta 2 \right).
\]

Write $\lambda = r e^{i\phi}$ with $\phi \in [-\pi, \pi)$ and we have $r \sim R$ and $- \frac \pi 2 \leq \phi \leq \frac \pi 2$. It follows that 
\[
\abs{G_1(x, \xi, - \lambda)} = \frac{e^{-\sqrt r \abs{x-\xi}\cos\frac\phi 2}}{2\sqrt r}
\]
and then we have
\begin{eqnarray}
\abs{G_1(x,\xi,-\lambda)G_2(y, \eta, \lambda)} & \leq & \frac{c_1}{R}\exp\left(-\sqrt R (y + \eta)\sin \frac \theta 2- \sqrt R \abs{x-\xi}\cos \frac \phi 2\right) \nonumber \\
& & + \frac{c_2}{R}\exp\left(-\sqrt R (\eta - y)\sin \frac \theta 2- \sqrt R \abs{x-\xi}\cos \frac \phi 2\right). \label{eqn:G1G2estimate1}
\end{eqnarray}
When $\abs{x-\xi}> 0$, since $\frac \phi 2 \in [-\pi/4, \pi/4]$ we have $\cos \frac \phi 2 \geq \frac{\sqrt 2}{2}$. Note that $\eta - y \geq 0$ by assumption and $\sin \frac \theta 2 \geq 0$, so we have 
\[
\abs{G_1(x,\xi, -\lambda) G_2(y, \eta, \lambda)} \leq \frac c R e^{- k \sqrt R}
\]
for some constant $k > 0$ when $R$ is large. On the other hand we have
\[
\abs{d\lambda} = R d\theta
\]
and so we have
\begin{eqnarray*}
\abs{\int_{C_1 \cup C_2} G_1(x,\xi, -\lambda)G_2(y,\eta, \lambda) d\lambda} & \leq & c \int_{[0, \pi/2]\cup [3\pi/2, 2\pi]} e^{-k \sqrt R} d\theta
\end{eqnarray*}
which converges to zero as $R \To \infty$.  

When $\abs{x-\xi} = 0$, we have $\eta -y > 0$. Let $c> 0$ be a constant and we show that 
\[
\lim_{R \To \infty}\int_{\theta_1}^{\theta_2} \exp\left(- c\sqrt R \sin \frac\theta 2\right) d\theta = 0
\]
on $C_1$. Since $\cos \frac \theta 2 \geq c_3 > 0$ on $C_1$ for some constant $c_3> 0$, we have
\begin{eqnarray*}
\int_{\theta_1}^{\theta_2} \exp\left( - c\sqrt R \sin\frac \theta 2\right) d\theta & \leq & \frac{1}{c_3}\int_{\theta_1}^{\theta_2}\cos\frac \theta 2 \exp\left(-c\sqrt R \sin \frac \theta 2\right) d\theta \\
& = & - \frac{2}{c c_3\sqrt R}\int_{\theta_1}^{\theta_2} d\exp\left(-c \sqrt R \sin \frac \theta 2\right) \\
& = & \frac{2}{c c_3\sqrt R} \left\{\exp\left(-c \sqrt R \sin \frac{\theta_1}{2}\right) - \exp\left(-c \sqrt R \sin \frac{\theta_2}{2}\right)\right\}
\end{eqnarray*}
which converges to zero as $R \To \infty$. A similar argument shows that the integral converges to zero on the arc $C_2$. Apply this convergence in equation (\ref{eqn:G1G2estimate1}) by taking $c = \eta + y$ and $c= \eta - y$ and then we have 
\[
\lim_{R \To \infty} \abs{\int_{C_1 \cup C_2} G_1(x, \xi, -\lambda)G_2(y, \eta, \lambda)d\lambda} = 0.
\] 
This finishes the proof of the claim on $C_1\cup C_2$. 

The identity (\ref{eqn:Cauchyidentity}) implies that 
\[
G(x, y, \xi, \eta) = \frac{1}{2\pi i}\int_{- \Gamma_1 \cup - \Gamma_2} G_1(x,\xi, -\lambda)G_2(y, \eta, \lambda) d\lambda.
\]
Let $\lambda = \frac 1 4 + s^2$ with $s\in (0, \infty)$. Then on $\Gamma_1$ and $\Gamma_2$ we have
\[
G_1(x,\xi,-\lambda) = \frac{e^{-\abs{x-\xi}\sqrt{s^2 + 1/4}}}{2\sqrt{s^2 + 1/4}}.
\]
On $\Gamma_1$ we have $\alpha = - i s$ and then the Green function $G_2$ is given by 
\[
G_2(y, \eta, \lambda) = -\frac{1}{2i s}\left(w(s,y) \bar w(s,\eta) + \bar a(s) \bar w(s,y) \bar w(s, \eta)\right).
\]
On $\Gamma_2$ we have $\alpha = i s$ and the Green function $G_2$ is given by
\[
G_2(y, \eta, \lambda) = \frac{1}{2i s}\left(\bar w(s, y) w(s, \eta) + a(s) w(s,y) w(s,\eta)\right).
\]
So we have
\begin{eqnarray*}
G(x, y, \xi, \eta) & = & \frac{1}{2\pi i}\int_0^\infty \frac{1}{-2i s}\left(w(s, y)  + \bar a(s) \bar w(s,y) \right)\bar w(s, \eta) \frac{e^{-\abs{x-\xi}\sqrt{s^2 + 1/4}}}{2\sqrt{s^2 + 1/4}} 2s ds  \\
& & + \frac{1}{2\pi i}\int_\infty^0 \frac{1}{2is} \left(\bar w(s, y) + a(s) w(s,y) \right)w(s,\eta) \frac{e^{-\abs{x-\xi}\sqrt{s^2 + 1/4}}}{2\sqrt{s^2 + 1/4}} 2s ds  \\
& = & \frac{1}{2\pi }\int_0^\infty \Re \left\{a(s) w(s,y)w(s,\eta) + \bar w(s, y) w(s,\eta)\right\} \frac{e^{-\abs{x-\xi}\sqrt{s^2 + 1/4}}}{\sqrt{s^2 + 1/4}} ds
\end{eqnarray*}
which gives formula (\ref{eqn:GreenRe}) when $y \leq \eta$. 

Since the real part of $a(s) w(s,y)w(s,\eta) + \bar w(s,y) w(s,\eta)$ is even in $s$, the kernel of the integral is symmetric in $y$ and $\eta$. Therefore the same formula holds when $y \geq \eta$. 
\end{proof}

\medskip

\section{Asymptotic expansions of the Green's function}

For each fixed $(x,y)\in \Sigma^2$, the Green's function $G(x,y,\xi,\eta)$ has the limit zero as $(\xi, \eta)$ diverges to infinity. In this section we determine the asymptotic expansions along various paths when $(\xi, \eta)$ diverges to infinity, see Theorems \ref{thm:asymGeta0}, \ref{thm:asymGetainfty} and \ref{thm:asymGreenrays}. For the techniques of asymptotic expansion of integrals, we refer to the book \cite{Wongbook}.

\smallskip
 
First we derive a new formula of the integrand in the Green's function. Let 
\begin{equation}\label{eqn:fsy}
f(s,y) = \hF{-\half, -\half + is, 1+ is; e^{-2y} }
\end{equation}
and
\begin{equation}\label{eqn:ksyeta}
k(s,y, \eta) = a(s) e^{-is(y + \eta)} f(s,y) f(s,\eta) + e^{is(y -\eta)}\bar f(s,y) f(s,\eta)
\end{equation}
for $s\geq 0$ and $y, \eta > 0$. Then the Green function in Theorem \ref{thm:Greenintegral} is given by
\begin{eqnarray*}
G(x, y, \xi, \eta) & = & \frac{e^{y + \eta}}{2\pi \sqrt{(e^{2 y}-1)(e^{2\eta}-1)}} \int_{0}^{\infty} \Re k(s,y, \eta) \frac{e^{-\abs{x-\xi}\sqrt{\frac 1 4 + s^2}}}{\sqrt {\frac 1 4 + s^2}} ds \\
& = & \frac{e^{y + \eta}}{\pi \sqrt{(e^{2 y}-1)(e^{2\eta}-1)}} \int_{-\infty}^{\infty} k(s,y, \eta) \frac{e^{-\abs{x-\xi}\sqrt{\frac 1 4 + s^2}}}{\sqrt {\frac 1 4 + s^2}} ds.
\end{eqnarray*}

\begin{lemma}\label{lem:ksyeta}
For $s\geq 0$ and $ y, \eta > 0$ we have
\begin{equation}
k(s, y,\eta) = \frac{\Gamma(3/2+ is)}{\Gamma(is)}\sqrt{\pi} e^{-\frac  y 2}\sinh^2( y) \hF{\frac 34 - \half is, \frac 3 4 + \half i s, 2; - \sinh^2  y} e^{-is\eta}f(s,\eta)
\end{equation}
and
\begin{eqnarray}
\Re k(s, y,\eta) & = & \frac{\pi}{2}s\left(s^2 + \frac 1 4\right)\tanh(\pi s) e^{-\frac{ y +\eta}{2}}\sinh^2 y\sinh^2\eta \nonumber \\
& & \times \hF{\frac 34 - \half is, \frac 3 4 + \half i s, 2; - \sinh^2  y}\hF{\frac 34 - \half is, \frac 3 4 + \half i s, 2; - \sinh^2 \eta}.
\end{eqnarray}
\end{lemma}
\begin{proof}
From Lemma \ref{lem:connectionformula} we have 
\begin{eqnarray*}
\frac{k(s, y,\eta)}{e^{-i s \eta}f(s,\eta)} & = & \Gamma(1-i s)\Gamma(3/2 + is)\left(\frac{e^{is y}f(-s, y)}{\Gamma(1-i s)\Gamma(3/2 + is )} - \frac{e^{-is y}f(s, y)}{\Gamma(1+is)\Gamma(3/2- is)}\right) \\
& = & \Gamma(1-is)\Gamma(3/2+ is)\frac{\sqrt{1-e^{-2 y}}}{\sqrt{2\pi}}(\cosh  y)^{-is}\tanh^{\frac 3 2}( y)\sin(i \pi s) \\
& & \times \hF{\frac 5 4+ \half is, \frac 3 4 + \half i s,2; \tanh^2 y}.
\end{eqnarray*}
By the transformation formula
\[
\hF{a,b,c; w} = (1-w)^{-a}\hF{a, c-b,c; \frac{w}{w-1}}\quad \text{with} \quad \abs{\arg(1-w)}< \pi,
\]
we have
\begin{eqnarray*}
\frac{k(s, y,\eta)}{e^{-is \eta}f(s,\eta)} & = & \Gamma(1-is)\Gamma(3/2+ is)\frac{\sqrt{1-e^{-2 y}}}{\sqrt{2\pi}}\sinh^{\frac 32}( y)\sin(i \pi s) \hF{\frac 3 4- \half is, \frac 3 4+ \half is, 2; - \sinh^2 y} \\
& = & \sqrt{\pi} \frac{\Gamma(3/2+is)}{\Gamma(is)}\frac{\sqrt{1-e^{-2 y}}}{\sqrt 2}\sinh^{\frac 3 2}( y)\hF{\frac 34 - \half is, \frac 3 4 + \half i s, 2; - \sinh^2  y} \\
& = & \sqrt{\pi} e^{-\frac  y 2}\frac{\Gamma(3/2 + is)}{\Gamma(i s)}\sinh^2( y)\hF{\frac 34 - \half is, \frac 3 4 + \half i s, 2; - \sinh^2  y}, 
\end{eqnarray*}
which gives us the formula of $k(s, y,\eta)$. 

For the real part of $k(s, y,\eta)$, we have 
\begin{eqnarray*}
2\Re k(s, y, \eta) & = & \sqrt{\pi}e^{-\frac  y 2}\sinh^2( y)\hF{\frac 34 - \half is, \frac 3 4 + \half i s, 2; - \sinh^2  y} \\
& & \times \left(\frac{\Gamma(3/2+ is)}{\Gamma(i s)}e^{-is \eta}f(s,\eta)+ \frac{\Gamma(3/2- is)}{\Gamma(-is)}e^{is \eta}f(-s, \eta)\right),
\end{eqnarray*}
and
\begin{eqnarray*}
& & \frac{\Gamma(3/2+ is)}{\Gamma(i s)}e^{-is \eta}f(s,\eta)+ \frac{\Gamma(3/2- is)}{\Gamma(-is)}e^{is \eta}f(-s, \eta) \\
& = & (- is)\Gamma(3/2-i s)\Gamma(3/2 + is)\left(\frac{e^{is\eta}f(-s,\eta)}{\Gamma(1-i s)\Gamma(3/2 + is )} - \frac{e^{-is\eta}f(s,\eta)}{\Gamma(1+is)\Gamma(3/2- is)}\right) \\
& = & (-is)\Gamma(3/2- is)\Gamma(3/2 + is)\frac{1}{\Gamma(1-is)\Gamma(is)}\sqrt{\pi}e^{-\frac y 2}\sinh^2(\eta)\hF{\frac 3 4 + \half i s, \frac 3 4 -\half is, 2; - \sinh^2 \eta} \\
& = & \frac{\Gamma(3/2+ is)\Gamma(3/2- is)}{\Gamma(is)\Gamma(-is)}\sqrt{\pi}e^{-\frac \eta 2}\sinh^2(\eta) \hF{\frac 3 4 + \half is, \frac 3 4 - \half is, 2; -\sinh^2 \eta} \\
& = & \sqrt{\pi} s\left(s^2 + \frac 1 4\right)\tanh(\pi s) e^{-\frac \eta 2}\sinh^2(\eta)\hF{\frac 3 4 + \half is, \frac 3 4 - \half is, 2; -\sinh^2 \eta}. 
\end{eqnarray*}
So they give us the desired formula of $\Re k(s, y,\eta)$.  
\end{proof}

\subsection{Asymptotic expansion at $\eta = 0$}
Recall Pochhammer's symbol
\[
(a)_n = \frac{\Gamma(a+n)}{\Gamma(a)} \quad\text{with}\quad  a \ne 0, -1, -2, \ldots
\]
and the psi function (or digamma function)
\[
\psi(z) = \frac{\Gamma'(z)}{\Gamma(z)} \quad \text{with}\quad  z \ne 0, -1, -2, \ldots.
\]

\begin{thm}\label{thm:asymGeta0}
For any fixed $(x, y)\in \Sigma^2$ and $\xi \ne x$, the Green function $G(x, y, \xi, \eta) $ has the following expansion near $\eta = 0$:
\begin{eqnarray}
G(x, y, \xi, \eta) & = & \frac{\eta^{\frac 3 2}}{\sqrt {2\pi} \sqrt{1-e^{-2y}}} \int_{-\infty}^{\infty} \frac{\Gamma(3/2+ is)}{\Gamma( is)} e^{- is y}f(s,y) \frac{1}{\sqrt{s^2 + \frac 1 4}}e^{-\abs{x-\xi}\sqrt{s^2 + \frac 1 4}} ds \nonumber \\
& & + O(\eta^{\frac 5 2}). \label{eqn:asymGeta0}
\end{eqnarray}
\end{thm}
\begin{remark}
Note that when $s\To \infty$ we have the asymptotic expansion
\begin{eqnarray*}
\frac{\Gamma(3/2+ is)}{\Gamma(i s)}f(s,y) & = & i^{\frac 3 2} \sqrt{1-e^{-2 y}} s^{\frac 3 2} \left(1+ O(s^{-1})\right),
\end{eqnarray*}
so the improper integral in Theorem \ref{thm:asymGeta0} does not converge when $\abs{x-\xi} = 0$.
\end{remark}

\begin{proof}
For any $c = a+ b + m$ with $m=2$, we recall the following formula for $z\in (0, 1)$:
\begin{eqnarray*}
\frac{1}{\Gamma(c)}\hF{a,b,c;z} & = & \frac{(m-1)!}{\Gamma(a+m)\Gamma(b+m)} \sum_{k=0}^{m-1}\frac{(a)_k(b)_k (m-k-1)!}{k!}(z-1)^k \\
& & - \frac{(z-1)^m}{\Gamma(a)\Gamma(b)}\sum_{k=0}^\infty \frac{(a+m)_k (b+m)_k}{k! (k+m)!} (1-z)^k \times \\
& &  \Big{\{}\log(1-z) - \psi(k+1) - \psi(k+m+1) + \psi(a+k+m) + \psi(b+k+m)\Big{\}}.
\end{eqnarray*}
It follows from the above formula and (\ref {eqn:fsy}) that 
\begin{eqnarray*}
f(s,\eta) & = & \frac{\Gamma(1+i s)}{\Gamma(3/2)\Gamma(3/2+ is)}\left(1 + \half \left(\half - is\right)(e^{-2\eta}-1)\right)  \\
& & - \frac{\Gamma(1+is)}{\Gamma(-1/2)\Gamma(-1/2+ is) 2}(e^{-2\eta}-1)^2\times \\
& & \left(\log(1-e^{-2\eta}) -\psi(1) - \psi(3) + \psi(3/2) + \psi(3/2 + is)\right)  + O(\eta^3) \\
& = & \frac{\Gamma(1+is)}{\Gamma(3/2)\Gamma(3/2+ is)}\left(1- \left(\half - i s\right)(\eta -\eta^2)\right) \\
& & - \frac{\Gamma(1+i s)}{\Gamma(3/2)\Gamma(3/2+ is)}\frac{1+4s^2}{8}\eta^2\times \\
& & \left(\log(1-e^{-2\eta}) -\psi(1) - \psi(3) + \psi(3/2) + \psi(3/2 + is)\right)  + O(\eta^3). 
\end{eqnarray*}
Here we have also used the identity
\[
\Gamma(-1/2)\Gamma(-1/2 + is) = \frac{16}{1+4s^2}\Gamma(3/2)\Gamma(3/2+ is).
\]
So we have
\begin{eqnarray*}
e^{-is \eta} f(s,\eta) & = & \frac{\Gamma(1+is)}{\Gamma(3/2)\Gamma(3/2+ is)}\left\{\left[1- is \eta -\half s^2 \eta^2\right]\left[1- \left(\half - is\right)(\eta - \eta^2)\right]\right. \\
& & \left. - \frac{1+4s^2}{8}\eta^2\left(\log(1-e^{-2\eta}) -\psi(1) - \psi(3) + \psi(3/2) + \psi(3/2 + is)\right)\right\}  + O(\eta^3) \\
& = & \frac{\Gamma(1+is)}{\Gamma(3/2)\Gamma(3/2+ is)}\left\{1-\half \eta + \half \left(1- i s + s^2\right) \eta^2 \right. \\
& & \left. - \frac{1+4s^2}{8}\eta^2\left(\log(1-e^{-2\eta}) -\psi(1) - \psi(3) + \psi(3/2) + \psi(3/2 + is)\right)\right\}  + O(\eta^3)
\end{eqnarray*}
and then
\begin{eqnarray*}
k(s,y, \eta) & = & \left(a(s) e^{-is y} f(s, y) + e^{i s  y}\bar f(s, y)\right) e^{-is \eta}f(s,\eta) \\
& = & \left(- \frac{\Gamma(1-is)}{\Gamma(3/2)\Gamma(3/2 - is)} e^{- is  y}f(s, y) + \frac{\Gamma(1+ is)}{\Gamma(3/2)\Gamma(3/2 + is)}e^{is y}\bar f(s, y)\right) \left\{\cdots \right\} + O(\eta^3),
\end{eqnarray*}
where the function in $\left\{\cdots\right\}$ has pure-imaginary value and is given by 
\[
\{\cdots \} = -i \left(\half s \eta^2 +\frac{1+4s^2}{8} \Im \psi(3/2 + i s) \eta^2\right).
\]
It follows that 
\[
\Re k(s, y, \eta) = - \Im \left\{\frac{\Gamma(1-is)}{\Gamma(3/2)\Gamma(3/2- is)}e^{-is y} f(s, y) \right\}\left(s + \left(s^2 + \frac 1 4\right) \Im \psi(3/2 + is)\right) \eta^2 + O(\eta^3),
\]
which can be further simplified to 
\[
\Re k(s, y, \eta) = {\sqrt \pi}\Re\left\{ \frac{\Gamma(3/2 + is)}{\Gamma( is)}e^{-i s y} f(s, y)\right\} \eta^2 + O(\eta^3).
\]
The expansion formula of $G(x, y, \xi, \eta)$ follows by combining the above identity with the expansion
\[
\frac{e^{\eta}}{\sqrt{e^{2\eta}-1}} = \frac{1}{\sqrt 2}\eta^{-\half} + O(\eta^\half).
\]
\end{proof}

\subsection{Asymptotic expansion at $\eta = \infty$} 
\begin{thm}\label{thm:asymGetainfty}
For fixed $(x, y) \in \Sigma^2$ and fixed $\xi > 0$, the Green function $G(x, y, \xi, \eta)$ has the following expansion as $\eta \To \infty$:
\begin{equation}\label{eqn:asymGetainfty}
G(x, y,\xi, \eta) = \frac{2}{\sqrt \pi}\frac{(e^y -1)^{2}}{e^{\frac y 2}\sqrt{(e^{2y} -1)}} \frac{e^{\half \eta}}{\eta^\half (e^{2\eta}-1)^\half}\left(1 + O(e^{-\frac 3 2\eta})\right).
\end{equation}
\end{thm}

We need some preparations and first consider the asymptotic expansion of the following integral as $\eta \To \infty$,
\begin{equation}\label{eqn:Ietainfty}
I(\eta) = \int_{-\infty}^\infty \left(a(s) e^{-i s y}f(s, y) + e^{i s y}\bar f(s, y)\right)\frac{e^{-A \sqrt{s^2 + \frac 1 4}}}{\sqrt{s^2 +\frac 1 4}} e^{-i s\eta}ds
\end{equation}
where $A = \abs{x-\xi} \geq 0$.  We may assume $\eta -y$ is bounded below, say $\eta - y \geq 2$. 

For fixed $\xi > 0$, define the following functions on the complex plane:
\begin{eqnarray}\label{eqn:qetainfty}
q(z) & = & - \ \frac{\Gamma(1- z)\Gamma(3/2 + z)}{\Gamma(1+ z)\Gamma(3/2- z)} e^{-z y} \hF{-\half, -\half + z, 1+ z; e^{-2 y}} \nonumber \\
& & + \ e^{z y} \hF{-\half, -\half - z, 1- z; e^{-2 y}},  
\end{eqnarray}
and
\begin{equation}\label{eqn:petainfty}
p(z) = q(z)\frac{e^{ i A\sqrt{z^2 - \frac 1 4}}}{\sqrt{z^2 - \frac 1 4}}.
\end{equation}
Following the proof of Lemma \ref{lem:ksyeta}, the function $q(z)$ can be simplified as
\begin{equation}
q(z) = \sqrt{\pi}e^{-\frac  y 2}\sinh^2 y \frac{\Gamma(z+3/2)}{\Gamma(z)}\hF{\frac 3 4 - \half z, \frac 3 4+ \half z,2; - \sinh^2 y}. \label{eqn:q1F}
\end{equation}

We specify the branch of $\sqrt{z^2 - \frac 1 4}$ we choose: write
\begin{eqnarray*}
z- \half = r_1 e^{i \theta_1} & \text{with} & r_1 \geq 0 \text{ and } \theta_1 \in [0, 2\pi) \\
z+ \half = r_2 e^{i \theta_2} & \text{with} & r_2 \geq 0 \text{ and } \theta_2 \in [-\pi, \pi),
\end{eqnarray*}
then we take
\[
\sqrt{z^2 - \frac 1 4} = \sqrt{r_1 r_2} e^{i \frac{\theta_1 + \theta_2}{2}}
\]
and it has the branch cut $(-\infty, \half] \cup [\half, \infty)$. In particular, when $z = i s$ with $s \in \Real$, we have $\theta_1 + \theta_2 = \pi$ and
\[
\sqrt{(is)^2 - \frac 1 4} = i \sqrt{s^2 + \frac 1 4}.
\]
It follows that 
\[
\left(a(s) e^{- is  y} f(s, y) + e^{is  y} \bar f(s, y) \right)\frac{e^{-A \sqrt{s^2 + \frac 1 4}}}{\sqrt{s^2 + \frac 1 4}} = i p(is)
\]
and then
\begin{equation}\label{eqn:Ietap}
I(\eta) = \int_{-i \infty}^{i \infty} p(z) e^{-\eta z} dz.
\end{equation} 

\begin{proposition}\label{prop:Ietarealintegral}
The integral $I(\eta)$ has the following form
\begin{eqnarray}
I(\eta) & = & 2\sqrt{\pi} e^{-\frac \eta 2}e^{-\frac  y 2}\sinh^2 y \int_{0}^\infty \frac{\Gamma(2 + t)}{\Gamma(1/2+t)}\hF{\frac 1 2 - \half t, 1 + \half t, 2; - \sinh^2 y} \nonumber \\
& & \times \frac{\cos\left(A \sqrt{t(t+1)}\right)}{\sqrt{t(t+1)}} e^{-\eta t}dt. \label{eqn:asymIetainfty}
\end{eqnarray}
\end{proposition}

\begin{proof}
We consider the contour integral in Figure \ref{fig:contourIeta}. The arc $C_\eps$ has the center $\half$ with radius $\eps > 0$ and $C_1, C_2$ have the center $0$ with radius $R>0$. Since $p(z) e^{-\eta z}$ is analytic in the region bounded by the closed contour $\Omega$, the Cauchy Theorem implies that 
\begin{equation}\label{eqn:ointpeetaz}
\oint_\Omega p(z) e^{-\eta z}dz = 0.
\end{equation}
\begin{center}
\begin{figure}[!htp]
\includegraphics[scale=0.6]{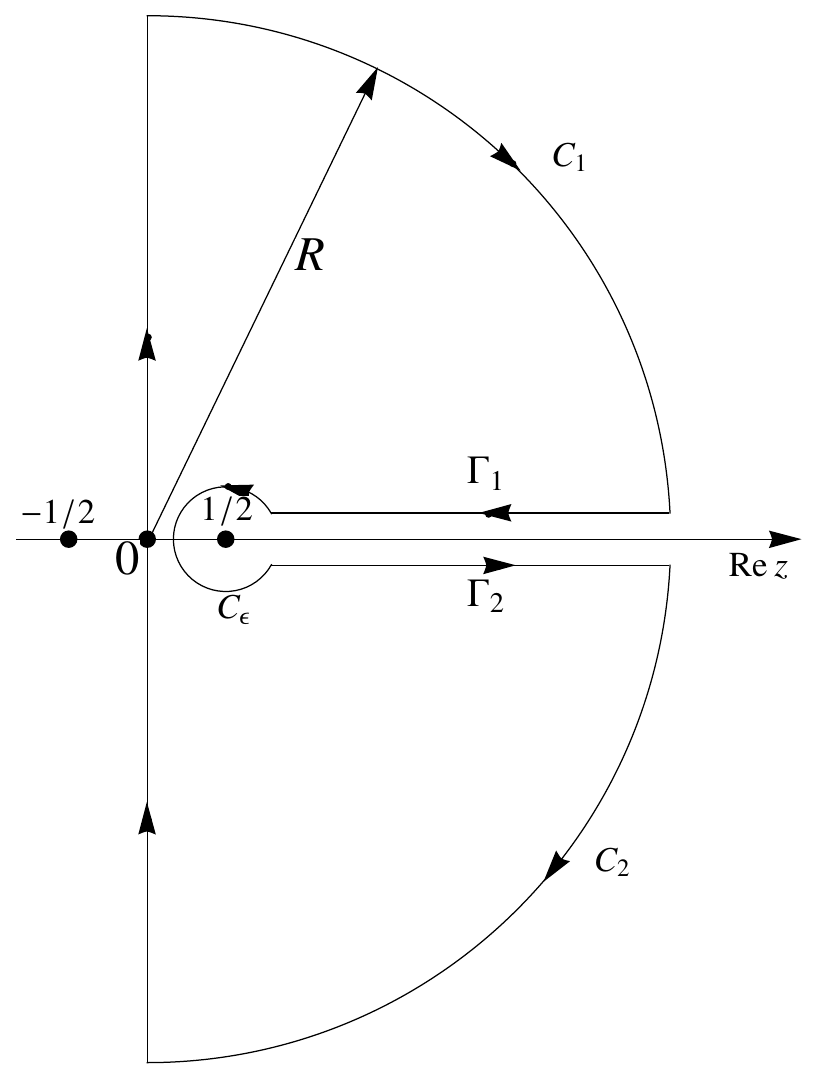}
\caption{The contour $\Omega$ used in the integral $I(\eta)$.}\label{fig:contourIeta}
\end{figure}
\end{center}

In the following we show that the integral $\int p(z) e^{-\eta z}dz$ vanishes on both $C_\eps$ and $C_1 \cup C_2$ as $\eps \To 0$ and $R \To \infty$ respectively. A similar argument as in the proof of Theorem \ref{thm:Greenintegral}  shows that 
\[
\lim_{\eps \To 0}\int_{C_\eps} p(z) e^{-\eta z}dz \To 0.
\]

\begin{claim}
We have
\begin{equation*}
\lim_{R\To \infty} \int_{C_1\cup C_2} p(z) e^{-\eta z}dz = 0.
\end{equation*}
\end{claim}

On the arcs $C_1\cup C_2$ we have $z = R e^{i\theta}$ with $\theta \in (0, \frac \pi 2]$ on $C_1$ and $\theta \in [\frac{3\pi}{2}, 2\pi)$ on $C_2$, see Figure \ref{fig:thetas}. Recall $z- \half = r_1 e^{i\theta_1}$ with $\theta_1\in [0, 2\pi)$ and $z + \half = r_2 e^{i\theta_2}$ with $\theta_2 \in [-\pi, \pi)$. On $C_1$ we have $0 \leq \theta_1 \leq \frac \pi 2 + \delta$ and $0\leq \theta_2< \frac \pi 2$ where $\delta > 0$ is a small number. On $C_2$ we have $\frac 3 2 \pi - \delta \leq \theta_1 \leq 2\pi$ and $-\frac \pi 2 < \theta_2 \leq 0$. So we have
\[
\sqrt{z^2 - \frac 1 4} = \sqrt{r_1 r_2} e^{i \frac{\theta_1 + \theta_2}{2}} = \sqrt{r_1 r_2}\left(\cos \frac{\theta_1 + \theta_2}{2} + i \sin \frac{\theta_1 +\theta_2}{2}\right)
\]
and 
\[
\Re i A \sqrt{z^2 - \frac 1 4} = - A \sqrt{r_1 r_2} \sin \frac{\theta_1 + \theta_2}{2}.
\]
\begin{center}
\begin{figure}[!htp]
\includegraphics[scale=0.7]{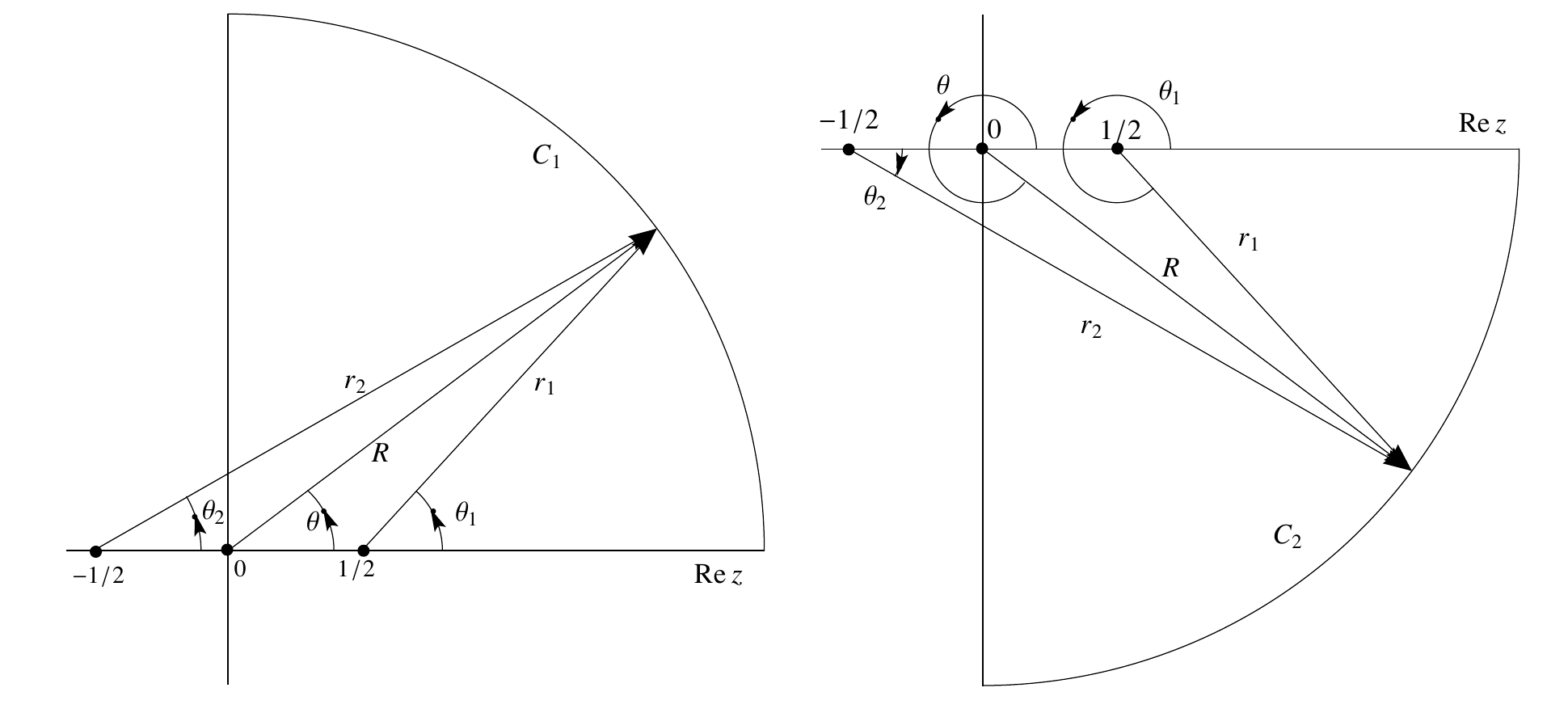}
\caption{Various angles in the contour integral along $C_1\cup C_2$.}\label{fig:thetas}
\end{figure}
\end{center}

Next we consider the asymptotic expansion of $q(z)$ for large $R > 0$. From the formula of $q(z)$ in (\ref{eqn:q1F}) and the asymptotic expansion of hypergeometric function in \cite[Section 15.12]{DLMF} we have
\begin{eqnarray*}
q(z) & = & 2\sqrt{\pi  y}e^{-\frac  y 2} \sqrt{\sinh y} \sqrt{z} \left(1 + O(z^{-1})\right)\left\{I_1(z y)\left(1+ O(z^{-2})\right) + \frac{I_0(z y)}{z}\left(A_1( y) + O(z^{-2})\right)\right\}
\end{eqnarray*}
where $I_\nu$ is the modified Bessel function and
\[
A_1( y) = \frac{3}{8 y} - \frac 3 8 \coth y < 0.
\]  

On the arc $C_1$, we have $\arg z = \theta \in [0, \pi/2]$ and so 
\begin{eqnarray*}
I_0(z y) & = & \frac{e^{z y} + i e^{-z y}}{\sqrt{2\pi  y z}}\left(1+ O(z^{-1})\right), \\
I_1(z y) & = & \frac{e^{z y} - i e^{-z y}}{\sqrt{2\pi  y z}}\left(1+ O(z^{-1}) \right).
\end{eqnarray*}
Since
\begin{eqnarray*}
\abs{e^{z y} \pm i e^{-z y}} \leq \abs{e^{z y}} + \abs{e^{-z y}} = e^{R  y \cos\theta} + e^{-R  y \cos\theta},
\end{eqnarray*}
it follows that
\begin{eqnarray*}
\abs{q(z)} & \leq & {c_1}\left(e^{R  y \cos\theta}+e^{-R  y \cos\theta}\right)
\end{eqnarray*}
for some constant $c_1 = c_1( y)$ when $R$ is large. Note that $\sqrt{r_1 r_2} \approx R$ when $R$ is large, so we have
\begin{eqnarray*}
\abs{p(z) e^{-\eta z} dz} & \leq & {c_1}\left(e^{R y \cos \theta} + e^{-R  y \cos\theta}\right) \frac{c_2 \exp\left(- A \sqrt{r_1 r_2} \sin \frac{\theta_1 + \theta_2}{2}\right)}{R}e^{-\eta R \cos \theta}  Rd\theta \\
& = & {c_1 c_2} \left\{\exp\left(R  y \cos \theta - A \sqrt{r_1 r_2}\sin\frac{\theta_1+\theta_2}{2} - \eta R \cos\theta\right)\right. \\
& & \left. + \exp\left(-R  y \cos \theta - A \sqrt{r_1 r_2}\sin\frac{\theta_1+\theta_2}{2} - \eta R \cos\theta\right)\right\}d\theta \\
& \leq & 2c_1 c_2 e^{-2R\cos \theta}d\theta.
\end{eqnarray*}
Note that $\frac{\theta_1+\theta_2}{2} \in [0, \pi]$ and we have assumed that $\eta -  y \geq 2$.  So we have
\begin{eqnarray*}
\abs{\int_0^{\pi/2} p(z) e^{-\eta z}dz} & \leq & 2c_1 c_2\int_0^{\pi/2} e^{-2R \cos\theta} d\theta \\
& \leq & 2c_1 c_2 \int_0^{\pi/4} e^{-2R\cos\theta} d\theta + 2\sqrt 2 c_1 c_2 \int_{\pi/4}^{\pi/2} e^{-2R \cos\theta} \sin \theta d\theta \\
& \leq & 2c_1 c_2 \int_0^{\pi/4} e^{-\sqrt 2 R }d\theta + \frac{\sqrt 2 c_1 c_2}{R}\int_{\pi/4}^{\pi/2} d e^{-2R \cos\theta} \\
& = & \frac{\pi}{2}c_1 c_2 e^{-\sqrt 2 R} + \frac{\sqrt{2} c_1 c_2}{R}\left(1 - e^{-\sqrt 2 R}\right)
\end{eqnarray*}
and hence 
\[
\lim_{R \To \infty} \int_{C_1} p(z) e^{-\eta z}dz = 0.
\]
On the arc $C_2$ we have different signs in the formulas of $I_0(z y)$ and $I_1(z y)$. The argument above carries through and the limit on $C_2$ also vanishes. This finishes the proof of the claim.

Thus, (\ref{eqn:Ietap}) and (\ref{eqn:ointpeetaz}) imply that
\[
I(\eta) = \int_{-\Gamma_1 \cup -\Gamma_2} p(z) e^{-\eta z} dz.
\]
Note that when $z = u \geq \half$ we have 
\[
\begin{array}{rcl}
\sqrt{z^2 - \frac 1 4} =\sqrt{u^2 - \frac 1 4} & \text{on} & \Gamma_1 \quad \text{and} \\
\sqrt{z^2 - \frac1 4} = -\sqrt{u^2 - \frac 1 4} & \text{on} & \Gamma_2. 
\end{array}
\]
It follows that 
\begin{eqnarray*}
I(\eta) & = & \sqrt{\pi}e^{-\frac  y 2}\sinh^2 y \int_{\half}^\infty \left\{ \frac{ \Gamma(3/2 + u)}{\Gamma(u)}\hF{\frac 3 4 - \half u, \frac 3 4 + \half u, 2; - \sinh^2 y}\right\} \\
& & \times \frac{e^{-\eta u}}{\sqrt{u^2 -\frac 1 4}}\left(e^{i A \sqrt{u^2 - \frac 1 4}} + e^{-i A \sqrt{u^2 - \frac 1 4}}\right) du \\
& = & 2\sqrt{\pi}e^{-\frac  y 2}\sinh^2 y \int_{\half}^\infty \left\{ \frac{\Gamma(3/2 + u)}{\Gamma(u)}\hF{\frac 3 4 - \half u, \frac 3 4 + \half u, 2; - \sinh^2 y}\right\} \frac{\cos\left(A \sqrt{u^2 - \frac 1 4}\right)}{\sqrt{u^2 -\frac 1 4}} e^{-\eta u}du \\
& = & 2 e^{-\half \eta }\sqrt{\pi}e^{-\frac  y 2}\sinh^2 y \int_{0}^\infty \left\{ \frac{\Gamma(2 + t)}{\Gamma(1/2+t)}\hF{\frac 1 2 - \half t, 1 + \half t, 2; - \sinh^2 y}\right\} \\
& & \times \frac{\cos\left(A \sqrt{t(t+1)}\right)}{\sqrt{t(t+1)}} e^{-\eta t}dt.
\end{eqnarray*}
This finish the proof of Proposition \ref{prop:Ietarealintegral}. 
\end{proof}

\begin{proposition}\label{prop:asymIetainfty}
For fixed $A= \abs{x-\xi}$ and $\eta>0$ large we have the asymptotic formula
\begin{equation}
I(\eta) \sim 2\sqrt{\pi} e^{-\frac 3 2 y}(e^{y}-1)^2 \eta^{-\half} e^{-\half \eta}\left(1 + O(\eta^{-1})\right)
\end{equation}
as $\eta \To \infty$
\end{proposition}

\begin{proof}
From Proposition \ref{prop:Ietarealintegral} and the formula of $q(z)$ in (\ref{eqn:q1F}) we have
\[
I(\eta) = 2\int_{\half}^\infty q(t) \frac{\cos\left(A\sqrt{t^2-\frac 1 4}\right)}{\sqrt{t^2 - \frac 1 4}}e^{-\eta t}dt
\]
At $t = \frac 1 2$, we have the following Taylor series
\[
q(t)\frac{\cos\left(A\sqrt{t^2-\frac 1 4}\right)}{\sqrt{t^2 - \frac 1 4}} =  e^{-\frac 3 2 y}(e^y -1)^2 (t-1/2)^{-\half} + O\left((t-1/2)^\half\right).
\]
To apply Watson's Lemma, see for example \cite{Wongbook}, we only need to show that there exist constants $M_1=
M_1(y)>0$ and $M_2=M_2(y) > 0$ independent of $t$ such that 
\[
\abs{q(t)\frac{\cos\left(A\sqrt{t^2 - \frac 1 4}\right)}{\sqrt{t^2 - \frac 1 4}}} \leq M_1 e^{M_2 t} \quad \text{for large } t> 0.
\]
Since $t\geq \half$, we have $\arg t = 0$ and the asymptotic expansion of the modified Bessel function has a simpler form, i.e., the $e^{-z\xi}$ term does not appear. So for large $t$ we have
\begin{eqnarray*}
q(t) & = & 2\sqrt{\pi  y} e^{-\frac  y 2} \sqrt{\sinh y} \sqrt{t} \left(1+ O(t^{-1})\right) \frac{e^{t  y}}{\sqrt{2\pi  y t}}\left(1+ O(t^{-1})\right) \\
& = & \sqrt 2 e^{-\frac  y 2} \sqrt{\sinh  y}e^{t y}\left(1+ O(t^{-1})\right).
\end{eqnarray*}
For example, we can take $M_1 = \sqrt 2 e^{-\frac  y 2}\sqrt{\sinh y}$ and $M_2 =  y$.

Now Watson's Lemma yields
\[
\int_{\half}^\infty q(t) \frac{\cos\left(A\sqrt{t^2-\frac 1 4}\right)}{\sqrt{t^2 - \frac 1 4}}e^{-\eta t}dt \sim \left(e^{-\frac 3 2 y}(e^ y -1)^2 \Gamma(1/2) \eta^{-\half} + O(\eta^{-\frac 3 2})\right) e^{-\half \eta}
\]
which gives the desired asymptotic formula of $I(\eta)$.
\end{proof}

From the asymptotic expansion of $I(\eta)$ in Proposition \ref{prop:asymIetainfty} we show Theorem \ref{thm:asymGetainfty}.

\begin{proof}[Proof of Theorem \ref{thm:asymGetainfty}]
Since 
\[
k(s, y, \eta) = \left(a(s) e^{-i s y} f(s, y) + e^{is  y}\bar f (s, y)\right)e^{-i s\eta} f(s,\eta)
\]
and $f(s,\eta) = 1 + O(e^{-2\eta})$ as $\eta \To \infty$, we have
\[
\int_{-\infty}^\infty k(s,y, \eta) \frac{\exp\left(-\abs{x-\xi}\sqrt{s^2 + \frac 1 4}\right)}{\sqrt{s^2 + \frac 1 4}} ds = I(\eta) + O(e^{-2\eta}).
\]
Using the asymptotic expansion of $I(\eta)$ in Proposition \ref{prop:asymIetainfty}, we have
\begin{eqnarray*}
G(x, y, \xi, \eta) & = & \frac{e^{y + \eta}}{\pi \sqrt{(e^{2y}-1)(e^{2\eta}-1)}}\left( I(\eta) + O(e^{-2\eta})\right) \\
& = & \frac{2}{\sqrt \pi}\frac{(e^y -1)^{\frac 3 2}}{e^{\frac y 2}(e^y +1)^\half} \frac{e^{\half \eta}}{\eta^\half (e^{2\eta}-1)^\half}\left(1 + O(e^{-\frac 3 2\eta})\right). 
\end{eqnarray*}
\end{proof}

\subsection{Asymptotic expansions along the rays $\eta = m\abs{\xi}$}

\begin{thm}\label{thm:asymGreenrays}
For each $m>0$, the Green function has the following asymptotic expansion along the ray $\eta = m\abs{\xi}$ with $\abs{\xi}\To \infty$:
\begin{eqnarray}
G(x,y,\xi, \eta) & = & \frac{\sqrt{m}\Gamma\left(\frac{3}{2} + \frac{m}{2\sqrt{m^2 +1}}\right) \sinh^{\frac 3 2}(y)}{8(m^2+1)^{\frac 3 4}\Gamma\left(1+\frac{m}{2\sqrt{m^2 +1}}\right)} \hF{\frac 3 4- \frac{m}{4\sqrt{m^2 + 1}}, \frac 3 4+ \frac{m}{4\sqrt{m^2 + 1}},2; -\sinh^2 y} \nonumber \\
& & \times \exp\left(\frac{\text{sgn}(\xi) x}{2\sqrt{m^2 + 1}}\right) e^{-\frac{\sqrt{m^2 + 1}}{2m}\eta}\left(\eta^{-1} + O\left(\eta^{-\half}\right)\right). \nonumber
\end{eqnarray}
\end{thm}

\begin{proof}
We first assume that $\xi > 0$ so that $\abs{\xi - x} = \xi -x$ for large $\xi$. From Lemma \ref{lem:ksyeta}, 
\[
k(s,y, \eta) = \sqrt \pi e^{-\frac  y 2}\sinh^2 ( y) \hF{\frac 3 4 - \half i s, \frac 3 4 + \half i s,2 ; - \sinh^2  y}\frac{\Gamma(3/2 + is)}{\Gamma(i s)} e^{-i s \eta}f(s,\eta). 
\]
We consider the real part of the following integral
\[
I(\xi) = \sqrt \pi \int_0^\infty e^{-\frac  y 2}\sinh^2 ( y) \hF{\frac 3 4 - \half i s, \frac 3 4 + \half i s,2 ; - \sinh^2  y}\frac{\Gamma(3/2 + is)}{\Gamma(i s)} e^{-i s \eta}\frac{e^{-\left(\xi-x\right)}\sqrt{s^2 + \frac 1 4}}{\sqrt{s^2 + \frac 1 4}}ds.
\]

Let 
\[
g( x, y, z) = \sqrt \pi e^{-\frac  y 2} \sinh^2( y) \hF{\frac 3 4 - \half i z, \frac 3 4+ \half i z,2; - \sinh^2 y} \frac{\Gamma(3/2+ i z)e^{x \sqrt{z^2 + \frac 1 4}}}{\Gamma(i z)\sqrt{z^2 + \frac 1 4}}
\]
and 
\[
\phi(z) = - \left(i m z + \sqrt{z^2 + \frac 14}\right)
\]
for $z\in \Cpx$. So the integral can be written as
\[
I(\xi) = \int_C g(x, y, z) e^{\xi \phi(z)} dz,
\]
where $C$ denotes the positive real axis. 

\begin{claim}
For $\xi \To \infty$, we have the  asymptotic expansion
\begin{eqnarray*}
\Re I(\xi) & = & e^{-\frac{\sqrt{m^2 +1}}{2}\xi}\left\{\frac{m\pi \Gamma\left(\frac{3}{2} + \frac{m}{2\sqrt{m^2 +1}}\right)}{2(m^2+1)^{\frac 3 4}\Gamma\left(1+\frac{m}{2\sqrt{m^2 +1}}\right)} \exp\left(\frac{x}{2\sqrt{m^2+1}}\right) \right. \\
& & \left. \times e^{-\frac y 2}\sinh^2(y)\hF{\frac 3 4- \frac{m}{4\sqrt{m^2 + 1}}, \frac 3 4+ \frac{m}{4\sqrt{m^2 + 1}},2; -\sinh^2 y} \xi^{-\half}+ O(1)\right\}
\end{eqnarray*}
\end{claim}

We show the above asymptotic expansion by the method of the steepest descent, see for example \cite[II.4]{Wongbook}. Write $z = u + i v$ and let $z_0 = - i v_0$ with
\[
v_0 =  \frac{m}{2\sqrt{m^2 + 1}}.
\]
Let $C_2$ be half of the hyperbola defined by the equation
\[
v = - \frac{m\sqrt{1+4(m^2+1)u^2}}{2\sqrt{1+m^2}} \quad \text{with} \quad u\geq 0.
\]
Denote $C_1$ the line segment from $z=0$ to $z_0$ on the $\Im z$-axis and $C_R$ the circular arc centered at the origin with radius $R$ connecting $C_2$ and the positive $\Re z$-axis, see Figure \ref{fig:contourhyperbola}.
\begin{center}
\begin{figure}[!htp]
\includegraphics[scale=0.8]{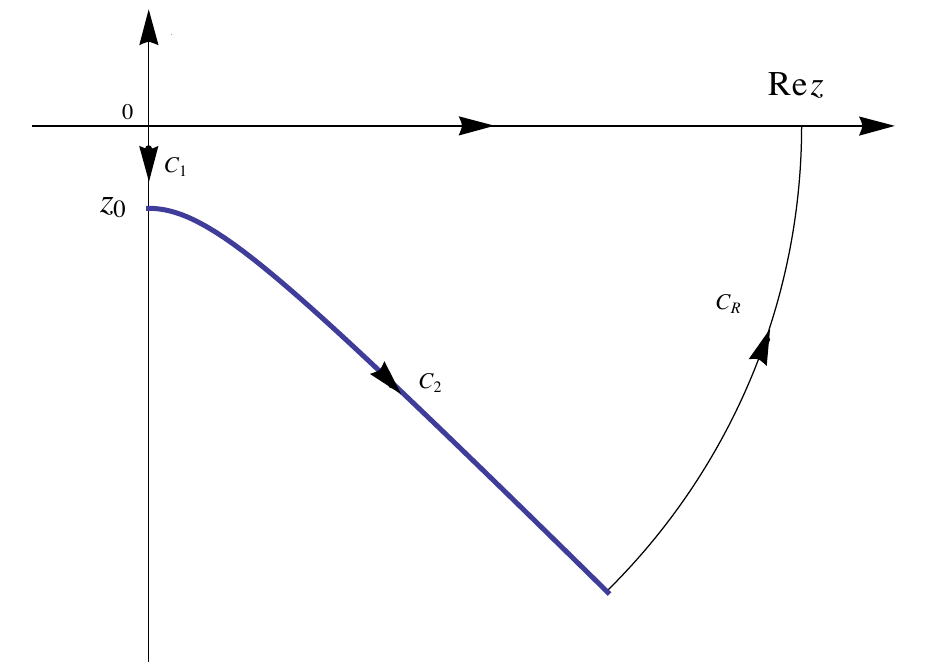}
\caption{Steepest descent path of the integral $I(\xi)$}\label{fig:contourhyperbola}
\end{figure}
\end{center}
Since $g(x, y, z) e^{\xi \phi(z)}$ is analytic in the region bounded by $\Re z$-axis, $C_1$, $C_2$ and $C_R$, the Cauchy Theorem implies that 
\[
I(\xi) = \lim_{R \To \infty} \int_{C_1 \cup C_2 \cup C_R} g(x,y,z)e^{\xi \phi(z)}dz.
\]
It can be shown that for any fixed $\xi > 0$ the integral along $C_R$ vanishes as $R \To \infty$. So we have
\[
I(\xi) = \int_{C_1} g(x, y, z) e^{\xi \phi(z)}dz + \int_{C_2}g(x, y, z) e^{\xi \phi(z)}dz.
\]
The integral along $C_1$ can be written as
\begin{eqnarray*}
I_1(\xi) & = & \int_{C_1} g(x, y, z) e^{\xi \phi(z)} dz \\
& = & - i \int_0^{v_0} g(x,y, -i v) e^{\xi \phi(-i v)}dv
\end{eqnarray*}
Note that $g(x, y, -i v)$ is a real-valued function and $\phi(-i v)$ is also a real-valued function since $v_0< \half$. So we have $\Re I_1(\xi) = 0$ for any $\xi > 0$. Next we consider the integral along $C_2$. We have
\[
z^2 + \frac1 4 = \frac{1}{4(m^2+1)} -(m^2 -1)u^2 - i \frac{m u \sqrt{1+4(m^2+1)u^2}}{\sqrt{m^2+1}}
\]
and
\[
\abs{z^2 + \frac 1 4} = \frac{1+ 4(m^2 +1)^2 u^2}{4(m^2 + 1)}.
\]
From the formula of the principal square root of a complex number 
\[
\sqrt{x+ i y} = \frac{1}{\sqrt 2}\left(\sqrt{\abs{x+iy}+ x}+ i \text{sgn}(y) \sqrt{\abs{x+i y} -x} \right),
\]
we have
\[
\sqrt{z^2 + \frac 1 4} = \frac{\sqrt{1+4(m^2+1)u^2}}{2\sqrt{m^2 + 1}} - i m u
\]
and
\[
\phi(z) =  - \frac{\sqrt{m^2 + 1}\sqrt{1+4(m^2 +1)u^2}}{2}.
\]
Along the curve $C_2$, let
\[
\tau = \phi(z_0) - \phi(z) = -\frac{\sqrt{m^2 + 1}}{2} - \phi(z)
\]
and we have $\tau \in [0, \infty)$. Around $z = z_0$ we have the following expansion
\begin{eqnarray*}
\phi(z) & = & \phi(z_0) - (m^2 + 1)^{\frac 3 2}(z-z_0)^2 + O\left((z-z_0)^3\right), \\
g(x, y,z) & = & b_0 + O(z-z_0)
\end{eqnarray*}
with
\begin{eqnarray*}
b_0 & = & \frac{m\sqrt \pi \Gamma\left(\frac{3}{2} + \frac{m}{2\sqrt{m^2 +1}}\right)}{\Gamma\left(1+\frac{m}{2\sqrt{m^2 +1}}\right)} \exp\left(\frac{x}{2\sqrt{m^2+1}}\right) \\
& & \times e^{-\frac y 2}\sinh^2(y)\hF{\frac 3 4- \frac{m}{4\sqrt{m^2 + 1}}, \frac 3 4+ \frac{m}{4\sqrt{m^2 + 1}},2; -\sinh^2 y}.
\end{eqnarray*}
It follows that 
\[
g(z) \frac{dz}{d\tau} = \frac{b_0}{2(m^2 +1)^{\frac 3 4}}\tau^{-\half} + O(1)
\]
and so we have the following asymptotic expansion as $\eta \To \infty$:
\begin{eqnarray*}
\Re I(\xi) = I_2(\xi) & = & e^{-\frac{\sqrt{m^2+1}}{2}\xi}\left(\frac{b_0}{2(m^2+1)^{\frac 3 4}}\Gamma(1/2) \xi^{-\half} + O(1)\right) \\
& = & e^{-\frac{\sqrt{m^2+1}}{2}\xi}\left(\frac{\sqrt{\pi} b_0}{2(m^2+1)^{\frac 3 4}} \xi^{-\half} + O(1)\right).
\end{eqnarray*}
This finishes the proof of the claim. 

The asymptotic expansion of the Green function follows from the one of $\Re I(\xi)$ for $\xi > 0$.  When $\xi < 0$ we have a similar expansion, except the factor $\exp\left(\frac{x}{2\sqrt{m^2 + 1}}\right)$ is replaced by 
\[
\exp\left(-\frac{x}{2\sqrt{m^2 + 1}}\right).
\]
So we have finished the proof of Theorem \ref{thm:asymGreenrays}. 
\end{proof}

\medskip{}

\section{The Martin kernel and Martin boundary}

In this section, we first determine the Martin compactification $\hat \Sigma$ of $\Sigma$ with respect to the operator $\calL = \Delta -1$, see Theorem \ref{thm:MartincompSigma}. Then we prove Theorem \ref{thm:uniquenessWpos} in the introduction at the end of the section. 

\smallskip

Fix a reference point $(x_0, y_0)\in \Sigma$, then the Martin kernel is given by
\begin{equation*}
K(x, y, \xi, \eta) =
\left\{
\begin{array}{cl} 
1, & \text{if }  (x, y)=(\xi, \eta)=(x_0, y_0) \\
& \\
\dfrac{G(x, y, \xi, \eta)}{G(x_0, y_0, \xi, \eta)}, & \text{otherwise}.
\end{array}
\right.
\end{equation*}
From the asymptotic expansion of the Green function along various paths in the previous section we can determine Martin kernel functions of such cases. 

For any $\xi \in \Real$, denote $\omega_{\xi} = (\xi, 0)$ on the boundary $\set{\eta = 0}$, and $\omega_\infty = \lim_{\eta \To \infty}(\xi_0, \eta)$ while fixing $\xi_0 > 0$. Recall the short notation of Gauss hypergeometric function
\[
f(s,y) = \hF{-\half, -\half + i s, 1+i s; e^{-2 y}}.
\]
Using Theorems \ref{thm:asymGeta0}, \ref{thm:asymGetainfty} and \ref{thm:asymGreenrays}, we  immediately obtain 
\begin{proposition}\label{prop:Martinkernel}
The limits of the Martin kernels as $(\xi, \eta)$ diverges to infinity are given by the following cases. 
\begin{enumerate}[(a)]
\item For any $\xi \in \Real$, when $\eta \To 0$ the Martin kernel is given by 
\begin{equation}
K(x, y, \omega_{\xi}) = \frac{\sqrt{1-e^{-2y_0}}}{\sqrt{1-e^{-2y}}} \frac{\int_{-\infty}^\infty \frac{\Gamma(3/2+ is)}{\Gamma(i s)}e^{- is y}f(s,y) \frac{1}{\sqrt{\frac 1 4 + s^2}}e^{-\abs{x-\xi}\sqrt{s^2 + \frac 1 4}}ds}{\int_{-\infty}^\infty \frac{\Gamma(3/2+ is)}{\Gamma(i s)}e^{- is y_0}f(s,y_0) \frac{1}{\sqrt{\frac 1 4 + s^2}}e^{-\abs{x_0-\xi}\sqrt{s^2 + \frac 1 4}}ds}
\end{equation}
with $x\ne \xi$.

\item When $\xi>0$ is fixed and $\eta \To \infty$, the Martin kernel is given by
\begin{equation}
K(x, y, \omega_\infty) = A(y_0)\frac{\left(e^{y}-1\right)^2}{e^{\frac y 2}\sqrt{e^{2 y}-1}},
\end{equation}
where $A(y_0)>0$ is a constant such that $K(x_0, y_0, \omega_\infty) = 1$.

\item Along the ray $\eta = \xi \tan\theta$ with $\theta\in (0, \pi/2)\cup (\pi/2, \pi)$, when $\eta\To \infty$ the Martin kernel is given by 
\begin{equation}
K(x, y, \theta) = A(x_0, y_0,\theta) \sinh^{\frac 3 2}(y)\hF{\frac 3 4 - \frac{\sin\theta}{4}, \frac 3 4 + \frac{\sin\theta}{4},2;-\sinh^2 y} \exp\left(\frac{\cos\theta}{2}x\right),
\end{equation}
where $A(x_0,y_0, \theta)$ is a constant such that $K(x_0,y_0,\theta) = 1$. 
\end{enumerate}
\end{proposition}

\begin{remark}
The function $K(0, y, \theta)$ is the unique (non-negative) solution, up to a constant multiple, to the differential equation 
\[
- w''(y) + P(y) w(y) = \lambda w(y) \quad \text{and} \quad w(0) = 0
\]
with $\lambda = \frac{1}{4}\cos^2\theta \in [0, \frac{1}{4}]$. It grows like $e^{\frac{\sin\theta}{2} y}$ as $y \To \infty$ if $\theta\ne 0$ or $\pi$. When $\theta=0$ and $\pi$, the function $K(0, y, 0)=K(0, y, \pi)$ grows like $\log\left(\cosh y \right)$ as $y \To \infty$.
\end{remark}

\begin{lemma}\label{lem:Kxixomegalim}
For any fixed $(x, y)\in \Sigma$, we have the following limits.
\begin{eqnarray*}
\lim_{\theta \To \frac\pi 2 }K(x, y, \theta) & = & K(x, y, \omega_\infty) \\
\lim_{\xi\To \infty} K(x, y, \omega_{\xi}) & = & \lim_{\theta\To 0}K(x, y, \theta) \\
\lim_{\xi\To -\infty} K(x, y, \omega_{\xi}) & = & \lim_{\theta\To \pi} K(x, y, \theta).
\end{eqnarray*}
\end{lemma}

\begin{proof}
The first identity follows from the defining equations of $K(x,y,\theta)$ and $K(x,y,\omega_\infty)$ in Proposition \ref{prop:Martinkernel}. Next we show the second identity. The third one follows by a similar argument. 

We assume that $\xi > x$ and denote $A = \xi -x> 0$. Let
\begin{eqnarray*}
I(A,y) = \int_0^\infty h(s,y) \frac{1}{\sqrt{s^2 + \frac 1 4}} e^{- A\sqrt{s^2 + \frac 1 4}}ds
\end{eqnarray*}
with
\[
h(s,y) = \Re\left\{\frac{\Gamma(3/2+i s)}{\Gamma(i s)} e^{- is y}f(s,y)\right\}.
\]
The Martin kernel can be written as 
\[
K(x, y, \omega_{\xi}) = \frac{\sqrt{1-e^{-2 y_0}}}{\sqrt{1-e^{-2 y}}}\frac{I(A, y)}{I(A,y_0)}.
\]
Denote by
\[
\Psi(x, y) = \lim_{\xi\To \infty} K(x, y, \omega_{\xi}).
\]
Since for any $\xi \in \Real$, $K(x, y, \omega_{\xi})$ is a positive $\calL$-harmonic function on $\Sigma^2$ and has value $1$ at $(x_0, y_0)$, so is $\Psi(x, y)$.

In the following we consider the asymptotic expansion of $I(A, y)$ as $A \to \infty$. Let $t = \sqrt{s^2 + \frac 1 4}$ and then $s= \sqrt{t^2 - \frac 1 4}$. The integral $I(A, y)$ can be written as
\[
I(A, y) = \int_{\frac 1 2}^\infty \frac{h(s,y)}{s} e^{- A t}dt.
\]
Suppose that $h(s,y)$ has the following expansion at $s = 0$:
\[
h(s,y) = s^{\alpha}\left(h_0(y)+ O(s)\right)
\]
with $h_0(y) \ne 0$ for $y > 0$, then we have
\[
\frac{h(s,y)}{s} = \left(t-\half\right)^{\frac{\alpha - 1}{2}}\left(h_0(y) + O\left(\sqrt{t-\half}\right)\right)
\]
and the following asymptotic expansion as $A \To \infty$ from Watson's lemma 
\begin{eqnarray*}
I(A, y) & = & e^{-\frac A 2} \frac{\Gamma((\alpha +1)/2)}{A^{\frac{\alpha +1 }{2}}}h_0(y)\left( 1+ O(A^{-1})\right) \\
& = & e^{\half x} e^{-\frac \xi 2}\frac{\Gamma((\alpha +1)/2)}{(\xi-x)^{\frac{\alpha +1 }{2}}}h_0(y)\left( 1+ O(A^{-1})\right).
\end{eqnarray*}
It follows that 
\[
\Psi(x, y) = e^{\half x}\frac{h_0(y)}{\sqrt{1-e^{-2 y}}}\frac{\sqrt{1-e^{-2 y_0}}}{h_0(y_0)e^{\half x_0}},
\]
and hence $\Psi(x, y) = K(x, y, 0)$.
\end{proof}
\begin{remark}
It can also be shown explicitly that the function $(1-e^{-2 y})^{-\half} h_0(y)$ solves the differential equation
\[
-w''(y) + P(y)w(y) = \frac 1 4 w(y) \quad \text{with}\quad w(0) = 0.
\]
\end{remark}

Recall the following definition of Martin compactification that is equivalent to the one with minimal Martin boundary in Section 2, see \cite[Definition 6.2]{Taylor}. Here we exclude the case where all positive solutions to $L w = 0$ are proportional. Note that a topological space is said to be $\sigma$-compact if it is the union of countably many compact subspaces. This property holds for any Riemannian manifold. In the notions of Martin compactification and Martin boundary, we drop the letter $L$.  
 
\begin{definition}\label{defn:MartincompTaylor}
Let $M$ be a complete smooth manifold that is $\sigma$-compact and $L$ be a second order strictly elliptic partial differential operator with smooth coefficients for which $L 1 \leq 0$. The Martin compactification  is the compactification $\hat M$ such that
\begin{enumerate}
\item the normalized Green functions, i.e., the Martin kernel functions $K(p, q)$, extend continuously to $\hat M$ for each $p \in M$; and
\item the extended functions separate the points of the ideal boundary $\partial \hat M = \hat M \backslash M$.
\end{enumerate} 
\end{definition}
The ideal boundary of the compactification satisfying condition (2) in Definition \ref{defn:MartincompTaylor} is the \emph{minimal Martin boundary} $\pd_e M$, see \cite[Remark I.7.7]{BorelJi}.

\begin{thm}\label{thm:MartincompSigma}
The Martin compactification of $\Sigma^2$ is homeomorphic to the half disk $\hat \Sigma = \set{(u,v)\in \Real^2 : u^2 + v^2 \leq 1, v\geq 0}$ and $\partial \hat \Sigma = \pd_e \Sigma$ is the minimal Martin boundary. Moreover we have
\begin{enumerate}
\item the real line $\Real\subset S_\infty(\Sigma)$ is identified with the interval $(-1,1)$ of $u$ with $v=0$, 
\item the vertical half line $\xi = \xi_0>0$ is identified with the point $(0,1)$, 
\item the asymptotic ray $\eta = \xi \tan \theta$ is identified with the semi-circle 
\[
\set{\omega_\theta = (\cos \theta, \sin \theta) : \theta \in (0, \pi/2)\cup (\pi/2, \pi)}\subset \pd \hat \Sigma,
\]
\item for any $\omega \in \partial \hat \Sigma \backslash \set{\omega_{\xi}}$ we have
\begin{equation}\label{eqn:Kxixomegaylim}
\lim_{(x, y)\To \omega}K(x, y, \omega_{\xi})  = 0
\end{equation}
\end{enumerate}
\end{thm}
\begin{proof}
Denote $S_\infty = S_\infty(\Sigma)$. First note that $\calL = \Delta - 1$ is coercive, see \cite[p. 498]{Ancona}. Since any $\omega_{\xi}$($\xi \in \Real$) can be approached by the geodesic $\gamma(t) = (\xi, t)$ with $t\in (0, 1]$ and the curvature of the subset $\Real \times (0, 1] \subset \Sigma$ is pinched by two negative constants,  it follows from \cite[Corollary 16]{Ancona} that $\omega_{\xi}$ is a minimal point on the Martin boundary and the limit in equation (\ref{eqn:Kxixomegaylim}) for $\omega \in S_{\infty} \backslash \set{\omega_{\xi}}$. In the Martin compactification each $\omega_{\xi}$ admits a basis of neighborhoods consisting of geodesic cones that agree with the basis of neighborhood in the topology of $\tilde{\Sigma}$. It follows that the collection of all $\omega_{\xi}$ that is $\Real \subset S_\infty$ embeds into the Martin compactification. 

The previous Lemma \ref{lem:Kxixomegalim} shows that Martin kernel functions $K(x, y, \xi, \eta)$ extend continuously to $\hat \Sigma$ for any $(x, y)\in \Sigma$, and it is clear to see that extended functions separate the points of $\hat \Sigma$. From Definition \ref{defn:MartincompTaylor} $\hat \Sigma$ is the Marin compactification and $\partial \hat \Sigma$ is the minimal Martin boundary.

In the following we show equation (\ref{eqn:Kxixomegaylim}) for other points $\omega \in \pd \hat \Sigma$. If $\omega = \omega_\theta$ is an asymptotic ray $y = x \tan \theta$ with $\theta \in (0, \pi)$, then we have $y \To \infty$ as $(x, y)$ approaches $\omega_\theta$. Note that $f(s, y)\To 1$ as $y \To \infty$. The rest part of the integral with $y$ in the numerator of $K(x, y, \omega_{\xi})$ is the Fourier coefficient of the following $L^2(-\infty, \infty)$-function in $s$:
\[
\frac{\Gamma(3/2 + is)}{\Gamma(is)} \frac{e^{-\abs{x-\xi}\sqrt{s^2 + \frac 1 4}}}{\sqrt{s^2 + \frac 1 4}}.
\]
So it converges to zero as $y \To \infty$ and thus we have
\[
\lim_{(x,y) \To \omega_\theta} K(x, y, \omega_{\xi}) = 0.
\]

 Next we consider the case when $y \To 0$ and $\abs{x}\To \infty$. We assume that $x\ne \xi$ when $x \To \pm \infty$. We have the following Taylor expansions:
\[
\sqrt{1-e^{-2y}} = \sqrt 2 y^{\half} + O\left(y^{\frac 3 2}\right)
\]
and
\[
\Re \frac{\Gamma(3/2 + is)}{\Gamma(is)} e^{- i s  y}f(s,  y) = \frac{\sqrt \pi}{2} s \left(s^2 + \frac 1 4\right)\tanh(\pi s)  y^2 + O( y^3).
\]
It follows that 
\begin{eqnarray*}
& & \int_{-\infty}^\infty \frac{\Gamma(3/2 + is)}{\Gamma(is)} e^{- i s  y}f(s,  y) \frac{e^{-\abs{x-\xi}}\sqrt{s^2 + \frac 1 4}}{\sqrt{s^2 + \frac 1 4}}ds \\
& = & 2\int_0^\infty \Re \frac{\Gamma(3/2 + is)}{\Gamma(is)} e^{- i s  y}f(s,  y) \frac{e^{-\abs{x-\xi}}\sqrt{s^2 + \frac 1 4}}{\sqrt{s^2 + \frac 1 4}}ds \\
& = & O( y^2)
\end{eqnarray*}
and hence $K(x, y, \omega_{\xi})$ converges to zero as $ y \To 0$. This finishes the proof of Theorem \ref{thm:MartincompSigma}.
\end{proof}

\begin{remark}
It follows from Theorem \ref{thm:geometriccompactification} that the Martin compactification $\hat \Sigma$ is homeomorphic to the geometric compactification $\tilde \Sigma$. The Martin kernel function is given by $K(x,y, 0)$ and $K(x,y,\pi)$ along the geodesics asymptotic to $\eta = \log \xi$ ($\xi \To \infty$) and $\eta = \log(-\xi)$ ($\xi \To -\infty$) respectively.
\end{remark}
\begin{remark}
From equation (\ref{eqn:Kxixomegaylim}) we know that 
\[
\lim_{(x, y)\To \omega_{\xi}} K(x, y, \omega_{\xi}) \ne 0
\]
if the limit exists. Otherwise, $K(x, y, \omega_{\xi})$ is the trivial function by the maximal principle which contradicts the fact that $K(x_0, y_0, \omega_{\xi}) = 1$. 
\end{remark}

We finish this section by showing Corollary \ref{cor:Dirichlet}.

\begin{proof}[Proof of Corollary \ref{cor:Dirichlet}]
We argue by contradiction. Assume that the Dirichlet problem at infinity is solvable for the Laplace operator $\Delta$. Let $f\geq 0$ be a continuous function on $S_\infty(\Sigma)$ such that $f = 0$ on the part $\set{v = 0}$ and $f>0$ on the part $\set{\omega_\theta : 0 < \theta < \pi}$. Then there is a harmonic function $F$ on $\Sigma^2$ such that the following properties hold: 
\begin{enumerate}
\item $F(x, y) = 0$ as $y$ approaches $0$;
\item $F(x,y) = f(\omega_{\frac \pi 2})$ as $(x,y)$ approaches infinity along the type (ii) geodesic, i.e., $x$ fixed and $y \To \infty$;
\item $F(x,y) = f(\omega_\theta)$ as $(x,y)$ approaches infinity along the type (iv) geodesic.
\end{enumerate}
It follows that $F$ is bounded and positive on $\Sigma^2$ by the maximum principle. Since $(\Sigma^2, ds^2)$ is conformal to the half-plane with the hyperbolic metric $g_H$, $F$ is also a positive harmonic function with respect to the metric $g_H$. Since $F$ vanishes on the boundary $\set{y=0}$, it is a positive constant multiple of $h_{\infty} = y$, the Martin kernel function at $\infty \in \pd_\Delta \Sigma$, see for example \cite[Remark 4.1]{Ancona}. This contradicts the previous conclusion that $F$ is bounded. 
\end{proof}

\begin{remark}\label{rem:Dirichletuniqueness}
From the proof of Corollary \ref{cor:Dirichlet}, it is easy to see that the Dirichlet problem at infinity for the Laplace operator $\Delta$ has a unique solution only when $f \in C^0\left(S_\infty (\Sigma)\right)$ is constant on the semi-circle $\set{\omega_{\theta} : 0 \leq \theta \leq \pi}$.
\end{remark}

\medskip

\section{Proof of Theorem \ref{thm:uniquenessWpos}}
 
\begin{proof}[Proof of Theorem \ref{thm:uniquenessWpos}]
We assume that $W$ is not the trivial solution so that $W$ is positive on $\Sigma$. From Theorem \ref{thm:Martinintegralrep}, there is a unique Borel measure $\nu$ on $\pd \hat \Sigma$ with $\nu(\pd \hat \Sigma) = 1$ such that 
\begin{equation}\label{eqn:Wintegral}
W(x, y) = \int_{\pd \hat \Sigma} K(x, y, \omega) d\nu(\omega).
\end{equation} 
By Proposition \ref{prop:Martinkernel} and Lemma \ref{lem:Kxixomegalim}, the proof now follows from the following
\begin{claim}
The measure $\nu$ is supported by the one-point set $\set{\theta = \pi/2}$, i.e., 
\[
\nu\left(\set{\theta = \frac \pi 2}\right) = 1.
\]
\end{claim}
We show the claim in two steps: first $\nu\left(\Real\right) = 0$ by the boundary condition $W(x,0) = 0$ for all $x\in \Real$, and then $\nu$ is concentrated at $\theta = \pi/2$ by the inequality (\ref{eqn:Wgrowth}). 

For the first step, we follow a similar argument in \cite[p. 65]{Brawn}. It is sufficient to show that if a finite Borel measure $\nu$ on $\Real$ has $\nu (\Real) > 0$, then
\[
h(x, y) = \int_{\Real} K(x, y, \omega_{\xi})d\nu (\xi) 
\]
cannot vanish on the $x$-axis. When $\nu(\Real)> 0$, there is a number $r> 0$ such that the $\nu$-measure of the interval $[-r,r]$ is positive. Let 
\[
h^*(x,y) = \int_{-r}^r K(x, y, \omega_{\xi})d\nu(\xi).
\]
Then $h^*$ is non-negative and $\calL$-harmonic on $\Sigma^2$. It is majorized by $h$, i.e., $h^*\leq h$ on $\Sigma^2$. Since $h^*(x_0, y_0) = \nu_0 = \nu([-r,r])>0$, $h^*$ is positive on $\Sigma^2$. From Theorem \ref{thm:MartincompSigma} we have $K(x, y, \omega_{\xi}) \To 0$ as $\abs{x} + y \To \infty$. Let $\set{(x_j, y_j)}_{j=1}^\infty$ be a sequence of points such that $\abs{x_j}+ y_j \To \infty$ as $j \To \infty$. For fixed $j$ since $K(x_j, y_j, \omega_{\xi})$ is continuous in $\xi$, it achieves the maximal value on $[-r,r]$. Let $\xi_j \in [-r, r]$ with
\[
K(x_j, y_j, \omega_{\xi_j}) = \max_{\xi \in [-r,r]}K(x_j, y_j,\omega_{\xi})
\]
and so we have
\[
h^*(x_j, y_j)\leq \nu_0 K(x_j, y_j, \omega_{\xi_j})
\]
It follows from the compactness of $[-r,r]$ that the limit of $K(x_j,y_j,\omega_{\xi_j})$ is zero as $j\To \infty$. So we have $h^*(x_j, y_j) \To 0$ as $j \To \infty$, i.e., $h^*(x, y) \To 0$ as $\abs{x} + y \To \infty$. It follows that $h^*(x,y)$ cannot vanish on the $x$-axis. Otherwise we would have $h^*(x, y) = 0$ everywhere that contradicts the fact $h^*(x_0,y_0)> 0$. So the function $h(x, y)$ cannot vanish on the $x$-axis. Therefore, the Borel measure $\nu$ in the claim must have $\nu\left(\Real\right) = 0$ because $W(x,0) = 0$ for all $x\in \Real$.

Now we proceed to the second step. Since $\nu(\Real) = 0$, the integral formula (\ref{eqn:Wintegral}) can be written as
\[
W(x, y) = \int_0^{\pi} K(x, y, \theta) d\nu(\theta).
\]
Inequality (\ref{eqn:Wgrowth}) is equivalent to 
\begin{equation}\label{eqn:Wgrowthde}
\pd_y W(a, y) - \half W(a, y)\geq 0 \quad \text{for all}\quad y \geq b.
\end{equation}
Let 
\[
J(x, y,\theta) = \pd_{y} K(x, y, \theta) - \half K(x, y, \theta). 
\]
Then, we have
\[
J(x,y,\theta) =J(0, y, \theta) e^{\frac{\cos\theta}{2} x}.
\]
It follows that 
\begin{eqnarray}\label{eqn:dWintegral}
\pd_y W(a, y) - \half W(a, y) & = & \int_0^\pi J(a, y, \theta) d\nu (\theta) \nonumber \\
& = & \int_0^\pi J(0, y, \theta) e^{\frac{a\cos\theta}{2}} d\nu(\theta).
\end{eqnarray}
We may assume that $a=0$, otherwise $e^{\frac{a\cos\theta}{2}}d\nu(\theta)$ defines another finite non-negative Borel measure $\tilde{\nu}$ on the semi-circle and the argument below also holds using the measure $\tilde{\nu}$.

From the asymptotic expansion of $K(0, y, \theta)$ in Remark 6.2, we have
\[
J\left(0, y, \frac\pi 2\right) > 0 \quad \text{and} \quad \lim_{y \To \infty} J\left(0, y,  \frac\pi 2\right) = 0.
\]
Also, for $\theta \ne \frac \pi 2$,
\[
J(0, y, \theta) \To -\infty \text{ at least linearly, as }y \To \infty.
\]
Assume that $\nu$ is not concentrated at $\theta = \frac \pi 2$. Then there exists some $\eps \in (0,\pi/2)$ such that $\nu(I_{\eps}) = \nu_0 > 0$ with $I_\eps = [0, \frac \pi 2 - \eps)\cup (\frac \pi 2 + \eps, \pi]$.
From the asymptotic behavior of $J(0, y, \theta)$, there exist $\delta> 0$ and $N > 0$ such that when $y > N$ we have
\begin{eqnarray*}
J(0, y, \theta) < - \delta & \text{for} & \theta \in I_\eps \\
J(0, y, \theta) < \nu_0 \delta & \text{for} & \theta \in [0, \pi] - I_\eps.
\end{eqnarray*}
The integral formula (\ref{eqn:dWintegral}) yields
\begin{eqnarray*}
\pd_y W(0, y) - \half W(0, y) & < & - \delta \nu_0 + \nu_0 \delta(1-\nu_0) \\
& = & - \nu_0^2 \delta 
\end{eqnarray*}
when $y > N$. This contradicts inequality (\ref{eqn:Wgrowthde}) for $y \geq \max\set{N, b}$ and we finish the proof.
\end{proof}

\medskip{}

\appendix

\section{Singular Sturm-Liouville problem and a Spectral  Theorem}\label{sec:STspectrum}

In this section we consider the Sturm-Liouville problem of the operator $- D_x^2 + P(x)$ on $(0, \infty)$, i.e., the spectrum of this operator. It is related to the Green's function $G_2(y,\eta, \lambda)$ of the operator $L_2 = D_y^2 - P(y) + \lambda$ in Section 4. For the basics of Sturm-Liouville problem, see for example \cite{Joergens} and \cite{Kodaira} and the references therein.

\smallskip

Recall the positive function
\[
P(x) = \frac{e^{4x} + 10 e^{2x}+1}{4(e^{2x}-1)^2}
\]
and consider the following 2nd order differential operator 
\begin{equation}\label{eqn:Adefn}
\left(A w\right)(x) = - w''(x) + P(x) w(x) 
\end{equation}
with the domain
\[
D(A) = \set{w \in L^2(0, \infty) : w \text{ and } w' \in \acloc(0, \infty), A w \in L^2(0, \infty)}.
\]
Here $\acloc(0,\infty)$ stands for the set of all locally absolutely continuous functions on $(0, \infty)$. Equation $A w = 0$ has the following two linearly independent solutions
\[
\frac{e^{x/2}}{\sqrt{e^{2x}-1}}\quad \text{and} \quad \frac{e^{2x}+1}{e^{x/2}\sqrt{e^{2x}-1}}.
\]
Since the second function is not square-integrable near $x=0$ and $x \To \infty$, we have limit-point case for $x =0$ and $x = \infty$. In particular, the operator $A$ with domain $D(A)$ is self-adjoint.

The differential equation $A w = \lambda w$ for any $\lambda \in \Cpx$ is solved in Appendix B and the two linearly independent solutions are denoted by $w_1(x)$ and $w_2(x)$, see Proposition \ref{prop:wfdmsolutions}. Note that a number $\lambda \in \Cpx$ is in the discrete spectrum $\sigma_d(A)$ of $A$ if the equation $A w = \lambda w$ has a solution $w\in L^2(0,\infty)$. 

\begin{thm}\label{thm:Adiscretespectrum}
The discrete spectrum $\sigma_d(A)$ of $A$ is empty.
\end{thm}
\begin{proof}
Since $A$ is self-adjoint, the spectrum $\sigma(A)$ is in $\Real$. We separate our discussion into different intervals of $\Real$. When $\lambda = 0$, the two linearly independent solutions of $A w = 0$ are not square integrable near $x = 0$. In the following we assume that $\lambda \ne 0$.
\smallskip

\noindent \textbf{Case 1}: $\lambda > \frac 1 4$. Write $\lambda = \frac 1 4 + s^2$ with $s>0$. Then the linearly independent solutions of $A w = \lambda w$ are given by
\begin{eqnarray*}
w_1(x) = \frac{e^{(1-i s)x}}{\sqrt{e^{2x}-1}}\hF{-\half, -\half + i s, 1+ i s; e^{-2x}}
\end{eqnarray*}
and $w_2(x) = \overline{w_1(x)}$, the complex conjugate. Since 
\[
\hF{-\half, -\half + i s, 1+ is; 0} = 1
\]
we have the following asymptotic expansion as $x\To \infty$:
\[
w_1(x) \sim \cos(sx) - i \sin(sx)
\]
and so no linear combination of $w_1(x)$ and $w_2(x)$ is square-integrable near $x = \infty$.

\noindent \textbf{Case 2}: $\lambda = \frac 1 4 - \alpha^2$ with $\alpha>0$ not an integer and $\alpha \ne \half$. When $\alpha$ is not a half integer, neither $w_1(x)$ nor $w_2(x)$ in Proposition \ref{prop:wfdmsolutions} is square integrable near $x = 0$. If $\alpha$ is a half integer, then $w_2(x)$ is given by the Jacobi polynomial. The solution $w_1(x)$ is not square integrable near $x=0$. When $x\To \infty$, we also have
\[
w_2(x) \sim e^{\alpha x}
\]
which is not square integrable. 

\noindent \textbf{Case 3}: $\lambda = \frac 1 4 - m^2$ with $m =0,1,2, \cdots$. The first solution $w_1(x)=P^m_\half\left(\frac{1+e^{-2x}}{1-e^{-2x}}\right)$ is not square integrable near $x=0$. The second solution is given by
\[
w_2(x) = Q^m_\half\left(\frac{1+e^{-2x}}{1-e^{-2x}}\right).
\]
with $w_2(0) = 0$. As $x\To \infty$, we have $w_2(x)$ approaches $Q^m_\half(1)$ which is unbounded by the asymptotic expansion $Q^m_\half(z)$ at $z = 1$. This finishes all the cases and proves Theorem \ref{thm:Adiscretespectrum}.
\end{proof}

For a complex number $z = re^{i\phi}$ with $r \geq 0$ and $\phi\in (-\pi, \pi]$, the principal square root is given by $\sqrt{z} = \sqrt{r} e^{i\frac{\phi}{2}}$. Recall the following functions in (\ref{eqn:as}) and (\ref{eqn:wsy}) of Theorem \ref{thm:Greenintegral}:
\begin{equation*}
a(s) = - \frac{\Gamma(1- i s)\Gamma(\frac 3 2 + i s)}{\Gamma(1+i s)\Gamma(\frac 3 2 - i s)}
\end{equation*}
and
\begin{equation*}
w(s, x) = \frac{e^{(1- i s)x}}{\sqrt{e^{2x}-1}} \hF{-\half, -\half + is, 1+ is; e^{-2x}}
\end{equation*}
for $s, x\in [0, \infty)$.

\begin{thm}\label{thm:expansionS}
For any function $h \in L^2(0, \infty)$ we have 
\begin{equation}
h(x) = \int_0^\infty K(x,y) h(y) dy,
\end{equation}
where the kernel is given by
\begin{equation}\label{eqn:Kxy}
K(x,y) = \frac{1}{\pi} \int_0^\infty \Re\set{a(s) w(s,x) w(s,y) + \bar w(s,x) w(s,y)}ds.
\end{equation}
\end{thm}
\begin{proof}
Suppose $\lambda$ is a complex number on the upper half plane, i.e., $\Im \lambda > 0$. Denote $\alpha = \sqrt{\frac 1 4 - \lambda}$ the principal square root of $\frac 1 4 -\lambda$, see also equation (\ref{eqn:alphaphi}). Since $\alpha$ is not a real number, by Proposition  \ref{prop:wfdmsolutions} the two linearly independent solutions are given by
\begin{eqnarray*}
w_1(\lambda, x) & = & e^{(1-\alpha)x}\left(e^{2x}-1\right)^{-\half}\hF{-\half, -\half +\alpha, 1+\alpha; e^{-2x}}\\
w_2(\lambda, x) & = & e^{(1+\alpha)x}\left(e^{2x}-1\right)^{-\half}\hF{-\half, -\half -\alpha, 1-\alpha; e^{-2x}}.
\end{eqnarray*}
The Wronskian of $w_1$ and $w_2$ is given by
\[
W(w_1, w_2) = w_1(\lambda, x)\pd_x w_2(\lambda, x) - w_2(\lambda, x) \pd_x w_1(\lambda, x) = 2\alpha.
\]

Let 
\begin{eqnarray*}
w_a(\lambda, x) & = & w_1 + a(\lambda) w_2 \\
w_b(\lambda, x) & = & w_1
\end{eqnarray*}
with
\[
a(\lambda) = - \frac{\hF{-\half, -\half + \alpha, 1+\alpha; 1}}{\hF{- \half, -\half -\alpha, 1-\alpha;1}} = - \frac{\Gamma(1+\alpha)\Gamma(\frac 3 2 - \alpha)}{\Gamma(1-\alpha)\Gamma(\frac 3 2 + \alpha)}
\]
then $w_a$ is square integrable near $x=0$. When $x$ approaches $\infty$ we have the following expansions,
\[
w_1(x) \sim e^{-\alpha x} \quad \text{and} \quad w_2(x) \sim e^{\alpha x}.
\]
Write $\alpha^2 = \frac 1 4 - \lambda = r e^{i \phi}$. We have $\Im(\alpha^2) < 0$ and then $\phi \in (-\pi, 0)$. It follows that $\frac{\phi}{2} \in (-\pi/2, 0)$ and $\alpha = \sqrt r \cos(\phi/2) + i \sqrt r \sin(\phi/2)$. So the real part $\Re \alpha> 0$. It follows that any square integrable solution near $x= \infty$ is a multiple of $w_b$.

We follow the construction in \cite[Section 12]{Joergens} (see also \cite{Kodaira}). Take a regular point $c\in (0, \infty)$ and consider the fundamental system $u_1(\lambda, x )$ and $u_2(\lambda, x)$ with the following boundary conditions: 
\begin{eqnarray*}
u_1(\lambda, c) = 1 & & \pd_x u_1(\lambda,c) =0; \\
u_2(\lambda, c) = 0 & & \pd_x u_2(\lambda,c) =1. 
\end{eqnarray*}
So the matrix $(m_{jk}(\lambda))$ is given by
\begin{eqnarray*}
m_{11}(\lambda) = w_a(\lambda, c) & & m_{12} (\lambda) = \pd_x w_a(\lambda, c), \\
m_{21}(\lambda) = w_b(\lambda, c) & & m_{22} (\lambda) = \pd_x w_b(\lambda, c)
\end{eqnarray*}
and
\begin{eqnarray*}
m(\lambda) & = & \det (m_{jk}) = W(w_a , w_b) \\
& = & W(w_1 + a(\lambda) w_2, w_1) \\
& = & - a(\lambda) W(w_1, w_2) \\
& = & 2\alpha \frac{\Gamma(1+\alpha)\Gamma(\frac 3 2 - \alpha)}{\Gamma(1-\alpha) \Gamma(\frac 3 2 + \alpha)}\\
& = & 2\alpha a(\lambda).
\end{eqnarray*}
The characteristic matrix $\left(G_{jk}(\lambda)\right)$ and the matrix function $\left(\rho_{jk}(\lambda)\right)$ are defined for $\lambda \in \Real$ as in \cite[Section 12]{Joergens}. If $\lambda \leq \frac 1 4$, then $\alpha \geq 0$ is a real number. So $(m_{jk})$ and then $(G_{jk})$ are real matrices and thus $\rho_{jk} = 0$ ($j,k = 1,2$). Note that the discrete spectrum $\sigma_d(A)$ is empty by Theorem \ref{thm:Adiscretespectrum}. 

Let 
\[
s= \sqrt{\lambda - \frac 1 4} > 0.
\]
Then we have $\alpha= - i s$ and 
\[
a(s^2 + 1/4) = a(s), \quad w_2(s^2 + 1/4,x) = w(s,x).
\]
For the fixed $c\in (0, \infty)$, let 
\[
b_1(s) = w(s,c) \quad \text{and}\quad b_2(s) = \pd_x w(s,x)|_{x= c}.
\]
Since $w_1(s^2 + \frac 1 4, x) = \overline{w_2(s^2 + \frac 1 4, x)} = \bar w(s,x)$, we have
\begin{eqnarray*}
\overline{b_1(s)} b_2(s) - b_1(s) \overline{b_2(s)} = W(w_1, w_2) = - 2 i s.
\end{eqnarray*}
Denote $m_{jk}(s^2+1/4)$ by $m_{jk}(s)$. Then we have
\begin{equation*}
\begin{pmatrix}
m_{11}(s) & m_{12}(s) \\
m_{21}(s) & m_{22}(s)
\end{pmatrix}
=
\begin{pmatrix}
\overline{b_1(s)} + a(s) {b_1(s)} & \overline{b_2(s)} + a(s) b_2(s) \\
\overline{b_1(s)} & \overline{b_2(s)}
\end{pmatrix},
\end{equation*}
and
\[
m(s)= m(s^2 + 1/4) =2 i s a(s).
\]
So the fundamental system is given by
\begin{equation*}
\begin{pmatrix}
u_1(s, x) \\
\\
u_2 (s, x)
\end{pmatrix}
=
\frac{i}{2s}
\begin{pmatrix}
-\overline{b_2(s)} & b_2(s) \\
& \\
\overline{b_1(s)} & - b_1(s) 
\end{pmatrix}
\begin{pmatrix}
w(s,x) \\
\\
\bar{w}(s,x)
\end{pmatrix}.
\end{equation*}
The characteristic matrix $\left(G_{jk}\right)$ is given by
\begin{eqnarray*}
G_{11} & = & \frac{i \left(\overline{b_1(s)}^2 + a(s) \abs{b_1(s)}^2\right)}{2sa(s)} \\
G_{12} & = & \frac{i\left(\overline{b_1(s)} \,\overline{b_2(s)} + a(s) b_1(s)\overline{b_2(s)}\right)}{2s a(s)} \\
G_{21} & = & \frac{i\left(\overline{b_1(s)}\, \overline{b_2(s)} + a(s)\overline{b_1(s)} {b_2(s)}\right)}{2s a(s)} \\
G_{22} & = & \frac{i \left(\overline{b_2(s)}^2+ a(s) \abs{b_2(s)}^2\right)}{2 s a(s)}
\end{eqnarray*}
and then 
\begin{eqnarray*}
\rho'_{11}(s) & =& \frac{1}{\pi}\left(\abs{b_1(s)}^2+ \Re \frac{b_1(s)^2}{\overline{a(s)}}\right) \\
\rho'_{12}(s) = \rho'_{21}(s) & = & \frac{1}{\pi}\Re \left(\overline{b_1(s)}b_2(s) + \frac{b_1(s)b_2(s)}{\overline{a(s)}}\right) \\
\rho'_{22}(s) & = & \frac{1}{\pi}\left(\abs{b_2(s)}^2 + \Re \frac{b_2(s)^2}{\overline{a(s)}}\right).
\end{eqnarray*}
Note that $\abs{a(s)}^2=1$, so we have
\[
\left(\rho'_{jk}(s)\right) = \frac{1}{2\pi}
\begin{pmatrix}
b_1 & \bar b_1 \\
b_2 & \bar b_2
\end{pmatrix}
\begin{pmatrix}
a & 1 \\
1 &  \bar a
\end{pmatrix}
\begin{pmatrix}
b_1 &  b_2 \\
\bar b_1 & \bar b_2
\end{pmatrix}.
\]
It follows that 
\begin{eqnarray*}
\sum_{j,k=1}^2 u_j(s,x) \rho_{jk}'(s) u_k(s,y)
& = & 
\begin{pmatrix}
u_1 & u_2 
\end{pmatrix}
\begin{pmatrix}
\rho'_{11} & \rho'_{12} \\
\rho'_{21} & \rho'_{22}
\end{pmatrix}
\begin{pmatrix}
u_1 \\
u_2 
\end{pmatrix} \\
& = & - \frac{1}{8 \pi s^2}
\begin{pmatrix}
w(s,x) & \bar w(s,x)
\end{pmatrix} \cdot \\
& & 
\begin{pmatrix}
-\bar b_2 & \bar b_1 \\
b_2 & - b_1 
\end{pmatrix}
\begin{pmatrix}
b_1 & \bar b_1 \\
b_2 & \bar b_2 
\end{pmatrix}
\begin{pmatrix}
a & 1 \\
1 & \bar a
\end{pmatrix}
\begin{pmatrix}
b_1 & b_2 \\
\bar b_1 & \bar b_2
\end{pmatrix} \cdot
 \\
& & \begin{pmatrix}
- \bar b_2 &  b_2 \\
\bar b_1 & - b_1
\end{pmatrix} 
\begin{pmatrix}
w(s,y) \\
\bar w(s,y)
\end{pmatrix}.
\end{eqnarray*}
A straightforward computation shows that the product of the five $2\times 2$-matrices above equals to
\begin{eqnarray*}
- 4 s^2
\begin{pmatrix}
a & 1\\
1 & \bar a
\end{pmatrix}.
\end{eqnarray*}
So we have
\begin{eqnarray*}
\sum_{j,k=1}^2 u_j(s,x) \rho'_{jk}(s) u_k(s,y) & = & \frac{1}{2\pi}
\begin{pmatrix}
w(s,x) & \bar w(s,x)
\end{pmatrix}
\begin{pmatrix}
a(s) & 1 \\
1 & \bar a(s)
\end{pmatrix}
\begin{pmatrix}
w(s,y) \\
\bar w(s,y)
\end{pmatrix} \\
& = &\frac{1}{\pi}\Re \left\{a(s)  w(s,x) w(s,y) + \bar w(s, x) w(s,y)\right\}.
\end{eqnarray*}
Now the eigenfunction expansion follows from \cite[12.10]{Joergens} or \cite[Theorem 1.13]{Kodaira}.
\end{proof}

\medskip{}

\section{A second order linear ordinary differential equation}

In this section we solve equation $A w = \lambda w$ given in (\ref{eqn:Adefn}) using special functions. First, let us recall the following special functions:
\begin{itemize}
\item The hypergeometric function 
\[
\hF{a,b,c; z} = \, _2F_1(a,b,c;z)
\]
solves the hypergeometric differential equation
\[
z(1-z)w''(z) + \left[c- (a+b+1)z\right] w'(z) - ab w(z) = 0
\]
with constants $a, b, c\in \Cpx$. 
\item The associated Legendre functions $P_\nu^\mu(z)$ and $Q_\nu^\mu(z)$ solve the  equation
\[
(1-z^2) w''(z) - 2z w'(z) + \left[\nu(\nu +1) - \frac{\mu^2}{1-z^2}\right] w(z) = 0
\]
with constants $\mu, \nu \in \Cpx$.
\item The Jacobi polynomial $\Jp_n^{(\alpha, \beta)}(z)$ is a solution of the  equation  
\[
(1-z^2) w''(z) + \left(\beta - \alpha - (\alpha + \beta + 2)z\right) w'(z) + n\left(n+ \alpha + \beta + 1\right) w(z) = 0
\]
with positive integer $n$ and constants $\alpha, \beta \in \Cpx$.
\end{itemize}
For detailed discussions and properties of these special functions, see \cite{DLMF} and \cite{OLBC}.

\begin{proposition}\label{prop:wfdmsolutions}
Let $\lambda \in \Cpx$. The general solution to the equation 
\begin{equation}\label{eqn:wdd}
w''(x) +\left(- \frac{e^{4x} + 10 e^{2x} +1}{4(e^{2x}-1)^2} + \lambda\right) w(x) = 0
\end{equation}
is given by $w = c_1 w_1 + c_2 w_2$, where $c_1, c_2\in \Cpx$ are constants, and $w_1$ and $w_2$ are defined for $x \in (0, \infty)$ as follows: 
\begin{enumerate}
\item If $\lambda = \frac 1 4 - \alpha^2$ with non-integer $\alpha$ and $\alpha \ne \half$, then we have
\begin{eqnarray*}
w_1(x) & = & \frac{e^{(1-\alpha)x}}{\left(e^{2x} - 1\right)^{\half}} \hF{- \half, -\half + \alpha, 1+\alpha; e^{-2x}} \\
w_2(x) & = & \frac{e^{(1+\alpha)x}}{\left(e^{2x} - 1\right)^{\half}} \hF{- \half, -\half - \alpha, 1-\alpha; e^{-2x}}.
\end{eqnarray*}

\item If $\lambda = \frac 1 4 - m^2$ with $m=0, 1,2, \ldots$, then we have
\begin{eqnarray*}
w_1(x) & = & P^m_\half\left(\frac{1+e^{-2x}}{1-e^{-2x}}\right) \\
w_2(x) & = & Q^m_\half\left(\frac{1+e^{-2x}}{1-e^{-2x}}\right),
\end{eqnarray*}
with $w_2(0) = 0$ and
\[
w_1(x) = \frac{\Gamma(3/2+m)}{\Gamma(3/2-m)\Gamma(1+m)}\frac{e^{(1-m)x}}{(e^{2x}-1)^\half}\hF{-\half, -\half +m, 1+m; e^{-2x}}.
\]
\end{enumerate}
\end{proposition}

\begin{remark}
(a) In Case (1), if $\alpha = i s$ is a pure imaginary number, then $w_1$ and $w_2$ are complex conjugate to each other. 

(b) In Case (1), if $\alpha = \half + m$ with $m = 1, 2, 3, \ldots$ then the hypergeometric function $\hF{-\half, -1-m, \half-m; z}$ reduces to the Jacobi polynomial $\Jp_n^{\alpha,\beta}$ as
\[
\hF{-\half, -1-m, \half -m; z} = \frac{(m+1)!}{(\frac 1 2- m)_{m+1}}\Jp_{m+1}^{-m - \half,-2}(1-2z).
\]
In particular we have $\hF{-\half, -1-m, \half - m; 1} = 0$. 

(c) In Case (1), two linearly fundamental solutions can also be given by associated Legendre functions $P^\alpha_\half$ and $Q^\alpha_\half$. Note that in this case $Q^\alpha_\half$ is a linear combination of $w_1(x)$ and $w_2(x)$ given in the proposition. 
\end{remark}

\begin{proof}
We start with the first case as $\lambda = \frac 1 4 - \alpha^2$. Let
\begin{equation}\label{eqn:transhy}
h(x) = e^{-\alpha x}\left(e^{2x} - 1\right)^{\half} w(x).
\end{equation}
Then equation (\ref{eqn:wdd}) can be written as 
\begin{equation}\label{eqn:hyperGFh}
\left(1-e^{2x}\right)h''(x) + 2(c-1-(a+b)e^{2x}) h'(x) - 4 ab e^{2x} h(x) = 0,
\end{equation}
that is the hypergeometric DE in $z = e^{2x}$ with $a = -\frac 1 2 + \alpha$, $b = -\frac 1 2$ and $c = \alpha + 1$. When $\alpha$ is not an integer, it has the fundamental solutions 
\begin{eqnarray*}
h_1(x) & = & e^{(1- 2\alpha) x}\hF{- \half, -\half + \alpha, 1+\alpha; e^{-2x}} \\
h_2(x) & = & e^{x} \hF{-\half, - \half - \alpha, 1-\alpha; e^{-2x}}
\end{eqnarray*}
with $e^{-2x}\in (0, 1]$. Since $a+b - c= -2 < 0$, $h_1$ and $h_2$ are absolutely convergent at $e^{-2x} = 1$, i.e., $x = 0$. By the transformation formula (\ref{eqn:transhy}), we have the fundamental solutions $w_1$ and $w_2$ in Case (1). 

When $\alpha=m$ with $m=0, 1, 2, \ldots$, let 
\[
z = \frac{1+e^{-2x}}{1- e^{-2x}} \in (1,\infty).
\]
Then equation (\ref{eqn:wdd}) is equivalent to the following Legendre differential equation,  
\begin{equation}\label{eqn:LegendreODE}
(1-z^2) w''(z) - 2 z w'(z) + \left(\frac 3 4 - \frac{m^2}{1-z^2}\right) w = 0.
\end{equation}
Its fundamental solutions are given by associated Legendre functions $P_{\half}^{m}(z)$ and $Q_{\half}^{m} (z)$ with the  Wronskian
\[
W\left(P_\half^m (z), Q_\half^m(z)\right) = \frac{(-1)^m \Gamma(3/2+m)}{(1-z^2)\Gamma(3/2-m)} \ne 0.
\]
So we have
\begin{eqnarray*}
w_1(x) & = & P_\half^m (z) = \frac{\Gamma(3/2+m)}{\Gamma(3/2-m)}P^{-m}_\half (z)\\
& = & \frac{\Gamma(3/2+m)}{\Gamma(3/2-m)}\frac{\sqrt{z+1}}{\sqrt 2 \Gamma(m+1)} \left(\frac{z-1}{z+1}\right)^\frac{m}{2} \hF{-\half, -\half + m, 1+m; \frac{z-1}{z+1}} \\
& = & \frac{\Gamma(3/2+m)}{\Gamma(3/2-m)\Gamma(1+m)}\frac{e^{(1-m)x}}{(e^{2x}-1)^\half}\hF{-\half, -\half +m, 1+m; e^{-2x}},
\end{eqnarray*}
which agrees with the first solution in the previous case. The function $Q_\half^m(z)$ can also be represented by hypergeometric function as
\begin{eqnarray*}
w_2(x) & = & Q_\half^m(z) \\
& = & \frac{(-1)^m \sqrt{\pi}\Gamma(3/2+m)}{2\sqrt 2}\frac{(z^2 - 1)^\frac{m}{2}}{z^{\frac 3 2 +m}}\hF{\frac m2 + \frac 3 4, \frac m 2 +\frac 5 4, 2; \frac{1}{z^2}}.
\end{eqnarray*}
Since $z\To \infty$ as $x \To 0$, we have $w_2(0) = 0$. 
\end{proof}

\medskip




\begin{thebibliography}{XXXX}

\bibitem[An1]{Ancona} A. Ancona, \emph{Negatively curved manifolds, elliptic operators, and the Martin boundary}, Ann. of Math. (2) \textbf{125} (1987), no. 3, 495--536.

\bibitem[An2]{Ancona2} A. Ancona, \emph{Covexity at infinity and Brownian motion on manifolds with unbounded negative curvature}, Rev. Mat. Iberoamericana \textbf{10} (1994) 189--220.

\bibitem[And]{Anderson} M. T. Anderson, \emph{The Dirichlet problem at infinity for manifolds of negative curvature}, J. Differential Geom. \textbf{18} (1983), no. 4, 701--721.

\bibitem[AS]{AndersonSchoen} M. T. Anderson and R. Schoen, \emph{Positive harmonic functions on complete manifolds of negative curvature}, Ann. of Math. (2) \textbf{121} (1985), no. 3, 429--461.

\bibitem[Ba1]{Ballmann1985} W. Ballmann, \emph{Nonpositively curved manifolds of higher rank}, Ann. of Math. (2) \textbf{122}(1985), no. 3, 597--609.

\bibitem[Ba2]{BallmannDirichlet} W. Ballmann, \emph{On the Dirichlet problem at infinity for manifolds of nonpositive curvature}, Forum Math. \textbf{1} (1989), no. 2, 201--213.

\bibitem[Ba3]{Ballmann} W. Ballmann, \emph{The Martin boundary of certain Hadamard manifolds}, Proceedings on Analysis and Geometry (Russian) (Novosibirsk Akademgorodok, 1999), 36--46, Izdat. Ross. Akad. Nauk Sib. Otd. Inst. Mat., Novosibirsk, 2000.

\bibitem[BL]{BallmannLedrappier} W. Ballmann and F. Ledrappier, \emph{The Poisson boundary for rank one manifolds and their cocompact lattices}, Forum Math. \textbf{6} (1994), no. 3, 301--313.

\bibitem[Bo1]{Borbely1} A. Borb\'{e}ly, \emph{A note on the Dirichlet problem at infinity for manifolds of negative curvature}, Proc. Amer. Math. Soc. \textbf{114} (1992), no. 3, 865--872. .

\bibitem[Bo2]{Borbely2} A. Borb\'{e}ly, \emph{The nonsolvability of the Dirichlet problem on negatively curved manifolds}, Differential Geom. Appl. \textbf{8} (1998), 217--237. 

\bibitem[BJ]{BorelJi} A. Borel and Lizhen Ji, \emph{Compactifications of symmetric and locally symmetric spaces}, Mathematics: Theory $\&$ Applications. Birkh\"{a}user Boston, Inc., Boston, MA, 2006. xvi+479 pp.

\bibitem[Br]{Brawn} F. T. Brawn, \emph{The Martin boundary of $\Real^n\times (0, 1)$}, J. London Math. Soc. (2) \textbf{5} (1972), 59--66. 

\bibitem[BS]{Burns} K. Burns and R. Spatzier, \emph{Manifolds of nonpositive curvature and their buildings}, Inst. Hautes \'{E}tudes Sci. Publ. Math. No. \textbf{65} (1987), 35--59. 

\bibitem[CL]{CaffarelliLittman} L. A. Caffarelli and W. Littman, \emph{Representation formulas for solutions to $\Delta u-u=0$ in $\Real^n$}, Studies in partial differential equations, 249--263, 
MAA Stud. Math., \textbf{23}, Math. Assoc. America, Washington, D.C., 1982.

\bibitem[CH]{CaoHesteady3d} H.-D. Cao and C. He, \emph{Infinitesimal rigidity of steady gradient Ricci soliton in three dimension}, preprint (2014), arXiv:1412.2714. 

\bibitem[Ch]{Choi} H. I. Choi, \emph{Asymptotic Dirichlet problems for harmonic functions on Riemannian manifolds}, Trans. Amer. Math. Soc. \textbf{281} (1984), no. 2, 691--716.

\bibitem[DLMF]{DLMF} NIST Digital Library of Mathematical Functions. \url{http://dlmf.nist.gov/}, Release 1.0.9 of 2014-08-29. Online companion to [OLBC]

\bibitem[DZ]{DingZhou} Q. Ding and D.T. Zhou, \emph{The existence of bounded harmonic functions on C-H manifolds}, Bull. Austral. Math. Soc., \textbf{53}(1996), 197--207.

\bibitem[EO]{EberleinONeill} P. Eberlein and B. O'Neill, \emph{Visibility manifolds}, Pacific J. Math. \textbf{46} (1973), 45--109.

\bibitem[FS]{Fisher-ColbrieSchoen}  D. Fischer-Colbrie and R. Schoen, \emph{The structure of complete stable minimal surfaces in $3$-manifolds of nonnegative scalar curvature}, Comm. Pure Appl. Math. \textbf{33} (1980), no. 2, 199--211.

\bibitem[Fu]{Fuerstenberg} H. F\"{u}rstenberg, \emph{A Poisson formula for semi-simple Lie groups}, Ann. of Math. (2) \textbf{77} (1963), 335--386.

\bibitem[Ha]{Hamiltonsurface} R. S. Hamilton, \emph{The Ricci flow on surfaces}, Mathematics and general relativity (Santa Cruz, CA, 1986), 237--262, Contemp. Math., \textbf{71}, Amer. Math. Soc., Providence, RI, 1988.

\bibitem[He]{Herglotz} G. Herglotz, \emph{\"{U}ber Potenzreihen mit positivem, reellen Teil im Einheitskreis}, Ber. Verh. Sachs. Akad. Wiss. Leipzig \textbf{63} (1911), 501--511.

\bibitem[Hs]{Hsu} E. P. Hsu, \emph{Brownian motion and Dirichlet problems at infinity}, Ann. Probab. \textbf{31} (2003), no. 3, 1305--1319.

\bibitem[HM]{HsuMarch} E. P. Hsu and P. Marsh, \emph{The limiting angle of certain Riemannian Brownian motion}, Comm. Pure Appl. Math. \textbf{38} (1985) 755--768.

\bibitem[Hu1]{Hua1} L.-K. Hua, \emph{On the theory of functions of several complex variables. I. A complete orthonormal system in the hyperbolic space of matrices} (Chinese. English summary),   J. Chinese Math. Soc.,   \textbf{2} (1953),  288--323. 

\bibitem[Hu2]{Hua2} L.-K. Hua, \emph{On the theory of functions of several complex variables. II. A complex orthonormal system in the hyperbolic space of Lie-hypersphere} (Chinese. English summary), Acta Math. Sinica \textbf{5} (1955), 1--25. 

\bibitem[Hu3]{Hua3} L.-K. Hua, \emph{On the theory of functions of several complex variables. III. On a complete orthonormal system in the hyperbolic space of symmetric and skew-symmetric matrices} (Chinese. Russian summary), Acta Math. Sinica \textbf{5} (1955), 205–-242. 

\bibitem[Ji]{Ji} R. Ji, \emph{The asymptotic Dirichlet problems on manifolds with  unbounded negative curvature}, preprint(2014), arXiv:1401.2476v1. 

\bibitem[Jo]{Joergens} K. J\"{o}rgens, \emph{Spectral theory of second-order ordinary differential operators,} Lectures delivered at Aarhus Universitet, 1962-1963.

\bibitem[Ko]{Kodaira} K. Kodaira, \emph{The eigenvalue problems for ordinary differential equations of the second order and Heisenberg's theory of $S$-matrices}, Amer. J. Math. \textbf{71} (1949), 921--945.

\bibitem[LT]{LiTam} P. Li and L.-F. Tam, \emph{Symmetric Green's functions on complete manifolds}, Amer. J. Math. \textbf{109 } (1987), no. 6, 1129--1154.

\bibitem[Ma]{Martin} R. S. Martin, \emph{Minimal positive harmonic functions}, Trans. Amer. Math. Soc. \textbf{49} (1941), 137--172.

\bibitem[Mu]{Muratastructure} M. Murata, \emph{Structure of positive solutions to $(- \Delta +V)u=0$ in $\Real^n$}, Duke Math. J. \textbf{53} (1986), no. 4, 869--943.

\bibitem[Ne]{Neel} R. Neel. \emph{Brownian motion and the Dirichlet problem at infinity on two-dimensional Cartan-Hadamard manifolds}, Potential Anal \textbf{41} (2014), 443--462.

\bibitem[OLBC]{OLBC} F. W. J. Olver, D. W. Lozier, R. F. Boisvert, and C. W. Clark, editors.\emph{ NIST Handbook of Mathematical Functions}, Cambridge University Press, New York, NY, 2010. Print companion to [DLMF].

\bibitem[Su]{Sullivan} D. Sullivan, \emph{The Dirichlet problem at infinity for a negatively curved manifold}, J. Differential Geom. \textbf{18} (1983), no. 4, 723--732.

\bibitem[Ta]{Taylor} J. C. Taylor, \emph{The Martin compactification associated with a second order strictly elliptic partial differential operator on a manifold $M$}, Topics in probability and Lie groups: boundary theory, 153--202, CRM Proc. Lecture Notes, \textbf{28}, Amer. Math. Soc., Providence, RI, 2001.

\bibitem[Ti1]{TitchmarshGreen} E. C. Titchmarsh, \emph{A relation between Green's functions}, J. London Math. Soc. \textbf{26} (1951). 31--36.

\bibitem[Ti2]{Titchmarsh} E. C. Titchmarsh,  \emph{Eigenfunction expansions associated with second-order differential equations}, Vol. 2. Oxford, at the Clarendon Press 1958 xi+404 pp.

\bibitem[Wo]{Wongbook} R. Wong, \emph{Asymptotic approximations of integrals}, Corrected reprint of the 1989 original. Classics in Applied Mathematics, \textbf{34} Society for Industrial and Applied Mathematics (SIAM), Philadelphia, PA, 2001.

\end{thebibliography}
\end{document}